\documentclass[12pt,letterpaper]{amsart}
\usepackage{amsmath,txfonts}
\usepackage{amssymb}
\usepackage{amsxtra}
\usepackage{amsthm, color}
\usepackage{txfonts}
\usepackage{graphicx}
\usepackage{times}
\usepackage{citeref}
\usepackage{tikz}
\usepackage{hyperref}
\usepackage{stmaryrd}
\usepackage[T3,T1]{fontenc}
\usepackage{pgfplots}

\usetikzlibrary{calc}

\usepackage{mathrsfs}
\usepackage{amsfonts}
\usepackage{amssymb}
\usepackage{ifthen}
\usepackage{graphicx}
\nonstopmode \numberwithin{equation}{section}

\newtheorem{thm}{Theorem}[section]
\newtheorem{lem}{Lemma}[section]
\newtheorem{cor}[thm]{Corollary}
\newtheorem{prop}[thm]{Proposition}

\newtheorem{step}{Step}[section]

\theoremstyle{definition}
\newtheorem{mlem}{Main lemma}[section]
\newtheorem{assertion}{Assertion}[section]
\newtheorem{cl}{Claim}[section]
\newtheorem{ca}{Case}[section]
\newtheorem{sca}{Subcase}[section]
\newtheorem{scl}{Subclaim}[section]
\newtheorem{conj}[thm]{Conjecture}
\newtheorem{fact}{Fact}[section]
\newtheorem{defn}[thm]{Definition}
\newtheorem{op}[thm]{Open Problem}
\newtheorem{prob}{Problem}[section]
\newtheorem{ques}{Question}[section]
\newtheorem{rem}[thm]{Remark}
\newtheorem{exam}[thm]{Example}

\numberwithin{equation}{section}

\newcounter {own}
\def\theown {\thesection       .\arabic{own}}

\newenvironment{pf}[1][]{%
 \vskip 3mm
 \noindent
 \ifthenelse{\equal{#1}{}}%
  {{\slshape Proof. }}%
  {{\slshape #1.} }%
 }%
{\qed\bigskip}

\newcounter{alphabet}
\renewcommand{\thealphabet}{\Alph{alphabet}}

\newenvironment{Thm}[1][]{\refstepcounter{alphabet}%
	\bigskip
	\noindent
	{\bf Theorem \thealphabet}%
	\ifthenelse{\equal{#1}{}}{}{ (#1)}%
	{\bf .} \itshape
}{\vskip 8pt}

\newenvironment{Lem}[1][]{\refstepcounter{alphabet}%
	\bigskip
	\noindent
	{\bf Lemma \thealphabet}%
	{\bf .} \itshape
}{\vskip 8pt}

\newenvironment{Core}[1][]{\refstepcounter{alphabet}%
	\bigskip
	\noindent
	{\bf Corollary \thealphabet}%
	{\bf .} \itshape
}{\vskip 8pt}

\newcommand{\T}{{\mathcal T}}

\def\be{\begin{equation}}
\def\ee{\end{equation}}

\newcommand{\ben}{\begin{enumerate}}
\newcommand{\een}{\end{enumerate}}

\newcommand{\blem}{\begin{lem}}
\newcommand{\elem}{\end{lem}}
\newcommand{\bthm}{\begin{thm}}
\newcommand{\ethm}{\end{thm}}
\newcommand{\bcor}{\begin{cor}}
\newcommand{\ecor}{\end{cor}}
\newcommand{\beg}{\begin{exam}}
\newcommand{\eeg}{\end{exam}}
\newcommand{\begs}{\begin{examples}}
\newcommand{\eegs}{\end{examples}}
\newcommand{\bdefe}{\begin{defn}}
\newcommand{\edefe}{\end{defn}}
\newcommand{\bprob}{\begin{prob}}
\newcommand{\eprob}{\end{prob}}
\newcommand{\bques}{\begin{ques}}
\newcommand{\eques}{\end{ques}}
\newcommand{\bei}{\begin{itemize}}
\newcommand{\eei}{\end{itemize}}
\newcommand{\bcon}{\begin{conj}}
\newcommand{\econ}{\end{conj}}
\newcommand{\bop}{\begin{op}}
\newcommand{\eop}{\end{op}}

\newcommand{\bas}{\begin{assertion}}
\newcommand{\eas}{\end{assertion}}

\newcommand{\bfa}{\begin{fact}}
\newcommand{\efa}{\end{fact}}

\newcommand{\bca}{\begin{ca}}
\newcommand{\eca}{\end{ca}}

\newcommand{\bst}{\begin{step}}
\newcommand{\est}{\end{step}}

\newcommand{\bsca}{\begin{sca}}
\newcommand{\esca}{\end{sca}}

\newcommand{\bcl}{\begin{cl}}
\newcommand{\ecl}{\end{cl}}

\newcommand{\bmlem}{\begin{mlem}}
\newcommand{\emlem}{\end{mlem}}

\newcommand{\bscl}{\begin{scl}}
\newcommand{\escl}{\end{scl}}

\newcommand{\bcons}{\begin{conjs}}
\newcommand{\econs}{\end{conjs}}

\newcommand{\bprop}{\begin{prop}}
\newcommand{\eprop}{\end{prop}}

\newcommand{\br}{\begin{rem}}
\newcommand{\er}{\end{rem}}
\newcommand{\brs}{\begin{rems}}
\newcommand{\ers}{\end{rems}}
\newcommand{\bo}{\begin{obser}}
\newcommand{\eo}{\end{obser}}
\newcommand{\bos}{\begin{obsers}}
\newcommand{\eos}{\end{obsers}}
\newcommand{\bpf}{\begin{pf}}
\newcommand{\epf}{\end{pf}}
\newcommand{\ba}{\begin{array}}
\newcommand{\ea}{\end{array}}
\newcommand{\beq}{\begin{eqnarray}}
\newcommand{\beqq}{\begin{eqnarray*}}
\newcommand{\eeq}{\end{eqnarray}}
\newcommand{\eeqq}{\end{eqnarray*}}

\newcounter{minutes}
\divide\time by 60
\newcounter{hours}
\multiply\time by 60 \addtocounter{minutes}{-\time}

\begin{document}

\def\thefootnote{}
\footnotetext{ \texttt{\tiny File:~\jobname .tex,
           printed: \number\year-\number\month-\number\day,
           \thehours.\ifnum\theminutes<10{0}\fi\theminutes}
} \makeatletter\def\thefootnote{\@arabic\c@footnote}\makeatother

\bibliographystyle{amsplain}
\title []
{The Hardy-Littlewood-type equivalence problem and its applications}

\author{Jiaolong Chen}
	\address{J. L. Chen, Key Laboratory of Computing and Stochastic Mathematics (Ministry of Education),
 School of Mathematics and Statistics, Hunan Normal University, Changsha, Hunan 410081, P. R. China}
	\email{jiaolongchen@sina.com}

\author{Shaolin Chen}
\address{S. L. Chen,    Center for Applied Mathematics of Guangxi, College of Mathematics and Statistics, Guangxi Normal University,
Guilin, Guangxi 541004, People's Republic of China} \email{slchen@gxnu.edu.cn}

\author{Hidetaka Hamada}
\address{H. Hamada, Industry-Academia Co-innovation and Research Promotion Headquarters,
Kyushu Sangyo University,
	3-1 Matsukadai 2-Chome, Higashi-ku, Fukuoka 813-8503, Japan.}
\email{hi.hamada01@gmail.com}

\author{Qianyun Li}
	\address{Q. Y. Li, Key Laboratory of Computing and Stochastic Mathematics (Ministry of Education), School of Mathematics and Statistics, 
Hunan Normal University, Changsha, Hunan 410081, P. R. China}
	\email{liqianyun@hunnu.edu.cn}



\begin{abstract}
The main purpose of this paper is to resolve the Hardy-Littlewood-type equivalence problem. In particular, we show that a Hardy-Littlewood-type theorem holds for arbitrary majorants in the case \(p\in(1,\infty)\), and that the equivalence holds in the case \(p=1\). For the proof, we employ new functional analytic approaches based on Fourier series, the Hilbert transform, and the uniform boundedness principle. Those proof methods are completely different from those of Dyakonov [Acta Math., 178 (1997), 143-167]  and Pavlovi\'c [Acta Math., 183 (1999), 141-143].
In addition, we present several representative applications that improve 
 a series of known results.
\end{abstract}

\maketitle \pagestyle{myheadings} \markboth{J. L. Chen, S. L. Chen,   H. Hamada and Q. Y. Li}{The Hardy-Littlewood-type equivalence problem and its applications}

\dedicatory{}


\section{Introduction and main results}\label{sec-1}

Let $\mathbb{D}$ denote the open unit disk in the complex plane $\mathbb{C}$, 
and let $\mathbb{T} := \partial \mathbb{D}$ be its boundary, the unit circle.
For a subset $\Omega\subset\mathbb{C}$,
we use $\mathcal{C}^n(\Omega)$ to denote the set of
all $n$-times continuously differentiable complex-valued functions  in $\Omega$,
where $n\in \mathbb{N}_{0}=\{0\}\cup \mathbb{N}$, and $\mathbb{N}$ denotes the set of all positive integers.
In particular, let $\mathcal{C}(\Omega):=\mathcal{C}^{0}(\Omega)$.

For $z = x + iy \in \mathbb{C}$, the formal complex derivatives are defined by
$$
\frac{\partial}{\partial z}
:= \frac{1}{2} \left( \frac{\partial}{\partial x} - i \frac{\partial}{\partial y} \right) 
\quad\text{and}\quad
\frac{\partial}{\partial \bar z}
:= \frac{1}{2} \left( \frac{\partial}{\partial x} + i \frac{\partial}{\partial y} \right).
$$

Given a  subdomain $\Omega\subset\mathbb{C}$ and a function $f\in \mathcal{C}^{1}(\Omega)$,
 the directional derivative of $f$ at the point $z \in \Omega$ in the direction $ \eta\in[0,2\pi) $ is defined by
$$
\partial_\eta f(z)
:= \lim_{r \to 0^+} \frac{f(z + r e^{i\eta}) - f(z)}{r}
= e^{i\eta} f_z(z) + e^{-i\eta} f_{\bar z}(z),
$$
where $f_z := \partial f / \partial z$ and $f_{\bar z} := \partial f / \partial \bar z$.
We set
$$
\mathscr{M}_{f}(z) := \max_{ \eta\in[0,2\pi)} |\partial_\eta f(z)|
= |f_z(z)| + |f_{\bar z}(z)|.
$$

 A twice continuously differentiable complex-valued function $f = u + iv$,
 where $u$ and $v$ are real-valued functions, 
 is called a \emph{complex-valued harmonic function} in the domain $\Omega \subset \mathbb{C}$ if both its real and imaginary parts satisfy
$$
\Delta u(z) = \Delta v(z)  = 0,~z=x+iy\in\Omega,
$$
 where $\Delta$ denotes the Laplacian operator
$$
\Delta := 4 \frac{\partial^2}{\partial z \, \partial \bar z}=\frac{\partial^{2}}{\partial x^{2}}+\frac{\partial^{2}}{\partial y^{2}}.
$$
By Lewy's theorem (see \cite[Theorem 1]{Lewy} or \cite[Section 2.2]{Duren2004}), we know that a complex-valued harmonic function $f$ 
is locally univalent and sense-preserving in $\Omega$ if and only if its Jacobian determinant
$$
J_f := |f_{z}|^2 - |f_{\overline{z}}|^2
$$
is strictly positive throughout $\Omega$. Equivalently, the second complex dilatation $\omega_f$ of $f$ satisfies
$$
|\omega_f|= \left| \frac{f_{\overline{z}}}{f_{z}} \right|< 1
$$
in $\Omega$.
Moreover, if $\Omega$ is simply connected, then $f$ admits a decomposition $f = h + \overline{g}$, where $h$ and $g$ are analytic in $\Omega$; this
representation is unique up to an additive constant (see \cite{Duren2004}).


For $\varphi\in\mathcal{C}(\mathbb{T})$, it is well known that if $f$  
satisfies the Dirichlet conditions:  (1)  $\Delta f=0$   in $ \mathbb{D} $, 
(2) $f\in \mathcal{C}(\overline{\mathbb{D}})$
and  (3)  $f\mid_{\mathbb{T}}=\varphi$, 
then  
$$
f(z)=\mathcal{P}[\varphi](z),~z\in \overline{\mathbb{D}},
$$ 
 where $\overline{\mathbb{D}}=\mathbb{D} \cup \mathbb{T}$
 (see  \cite[p. 13 and 15]{ABR92}). 
 Here and hereafter,
$\mathcal{P}[\varphi]$ stands for the  {\it Poisson integral} of $\varphi$ in $\overline{\mathbb{D}}$, i.e.,
\begin{equation*} 
\mathcal{P}[\varphi](z) =\begin{cases}
\displaystyle \frac{1}{2\pi} \int_{0}^{2\pi}\mathcal{P}(z,e^{i\eta})\varphi(e^{i\eta}) d\eta, \;\;  &\text{for}  \;z\in \mathbb{D}, \\
\displaystyle \varphi(z),\;  &\text{for}  \; z\in \mathbb{T} ,
\end{cases}
\end{equation*}
where 
$$\mathcal{P}(z,e^{i\eta})=\frac{1-|z|^2}{|z-e^{i\eta}|^2} $$
is  the  {\it Poisson kernel}  with respect  to $\Delta$, satisfying
\begin{equation}\label{ceq-1.5}
 \frac{1}{2\pi}\int_{0}^{2\pi}\mathcal{P}(z,e^{i\eta})d \eta =1. 
 \end{equation}

We denote by $\mathscr{H}(\mathbb{D})$ the class of all complex-valued harmonic functions from  $\mathbb{D}$ into  $\mathbb{C}$.
Moreover, we use $\mathscr{A}(\mathbb{D})$ to denote the set of all analytic functions of $\mathbb{D}$ into $\mathbb{C}$.

For convenience, we make a notational convention.
Throughout this paper,
 we use the symbol $C$ to denote various positive
constants, whose values may change from one occurrence to another.
Also we denote by $C=C(a_{1},a_{2},\ldots)$ a constant that depends only on the given parameters $a_{1}$, $a_{2}$, $\ldots$,
and whose value may vary from one occurrence to another.

\subsection*{The Hardy-Littlewood-type equivalence problem}
 A continuous increasing function $$\omega:[0,\infty)\mapsto[0,\infty)$$ with $\omega(0)=0$ is called a majorant if
$\omega(t)/t$ is non-increasing for $t\in(0,\infty)$ (see \cite{Dy1,Dy2,P}). For $\delta_{0}>0$ and $0<\delta<\delta_{0}$, we consider the
following conditions on a majorant $\omega$:
\be\label{eq2x}
\int_{0}^{\delta}\frac{\omega(t)}{t}\,dt\leq\, C\omega(\delta)
\ee
and
\be\label{eq3x}
\delta\int_{\delta}^{\infty}\frac{\omega(t)}{t^{2}}\,dt\leq\, C
\omega(\delta),
\ee
where $C$ denotes a positive constant. A majorant $\omega$ is henceforth said to be fast (resp. slow) if condition (\ref{eq2x}) (resp. (\ref{eq3x}) ) is fulfilled.
In particular, a majorant $\omega$ is said to be  regular if it satisfies the
conditions (\ref{eq2x}) and (\ref{eq3x}) (see \cite{Dy1,Dy2}).

Given a majorant $\omega$ and a subset  $\Omega$ of
$\mathbb{C}$, a function $f$ of  $\Omega$ into $\mathbb{C}$ is
said to belong to the  {\it Lipschitz  space}
$\Lambda_{\omega }(\Omega)$
 if there is a positive constant $C$ such that, for $z_{1},z_{2}\in\Omega$,
 \be\label{rrt-1}
 |f(z_{1})-f(z_{2})| \leq\,C\omega\left(|z_{1}-z_{2}|\right).
\ee
The Lipschitz constant of $f\in \Lambda_{\omega}(\Omega)$ is defined as follows: 
$$
\|f\|_{\Lambda_{\omega}(\Omega),s} 
:= \sup_{\substack{z_1,z_2\in\Omega\\ z_1\neq z_2}} \frac{|f(z_1) - f(z_2)|}{\omega(|z_1 - z_2|)}  .
$$
 Note that if $\Omega$ is a proper subdomain of $\mathbb{C}$ and 
$f\in \Lambda_{\omega }(\Omega)$, then $f$ is continuous on
$\overline{\Omega}$ and (\ref{rrt-1}) holds for $z,w \in \overline{\Omega}$  
(see \cite{Dy1} or \cite[Exercise  4.13]{Rudin1976}).
Furthermore,  for  a Borel function $f:\Omega\rightarrow\mathbb{C}$ with
   $\Omega=\mathbb{D}$,   $ \mathbb{T}$ or   $ \overline{\mathbb{D}}$,
     $f$ is said to belong to the {\it Lipschitz type space} $\Lambda_{\omega,p}(\Omega)$ ($p\in(0,\infty]$)
  if there exists a positive constant $C$ such that, for all  $z_1,z_2 \in \Omega$,
  $$
\mathscr{L}_p[f](z_1,z_2) \leqslant C\omega(|z_1 - z_2|),
 $$ 
where
\[
\mathscr{L}_p[f](z_1,z_2) =
\begin{cases}
	\displaystyle \biggl( \int_{0}^{2\pi} \bigl|f(e^{i\eta} z_1) - f(e^{i\eta} z_2)\bigr|^p d\eta \biggr)^{\frac{1}{p}}, & \text{if } p \in (0,\infty),\\[6pt]
	\bigl|f(z_1) - f(z_2)\bigr|, & \text{if } p = \infty 
\end{cases}
\]
(see \cite{P-08,Pav-2019}).
 Also, for  $f\in\Lambda_{\omega,p}(\Omega)$, 
 the Lipschitz type constant of $f$ is defined as follows:
\begin{equation}\label{ceq-1.2}
\|f\|_{\Lambda_{\omega,p}(\Omega),s} := \sup_{\substack{z_1,z_2\in \Omega\\ z_1\neq z_2}} \frac{\mathscr{L}_p[f](z_1,z_2)}{\omega(|z_1 - z_2|)} .
\end{equation}

Obviously, if $f\in \Lambda_{\omega,\infty}(\Omega)$ with $\Omega=\mathbb{D}$,  $ \mathbb{T}$ or $\overline{\mathbb{D}}$, 
then  $f\in \mathcal{C}(\Omega)$ and  
$$
\|f\|_{\Lambda_{\omega,\infty}(\Omega),s} = \|f\|_{\Lambda_{\omega}(\Omega),s}.
$$
If we let
\begin{equation} \label{ceq-1.3}
 \Delta_hf(e^{i\eta })=f\big(e^{i(\eta+h)}\big)-f(e^{i\eta}) 
\end{equation}
for any $h,\eta \in \mathbb{R}$ and $f\in\Lambda_{\omega,p}( \mathbb{T})$ with $p\in(0,\infty]$,
then we see   that
\begin{equation}\label{ceq-1.4}
 \|f\|_{\Lambda_{\omega,p}( \mathbb{T}),s}
 =
 \sup_{0<|h|\leq\pi}
 \frac{\|\Delta_hf\|_{L^p( \mathbb{T})}}
 {\omega(|e^{ih}-1|)}.
\end{equation}
For a more comprehensive treatment of the spaces $\Lambda_{\omega,p}(\mathbb{D})$, 
we refer the reader to \cite{Chen2025, Dy1, Dy2, P-08,Pav-2019}.

 In particular, for a domain $\Omega \subset \mathbb{C}$, we say that a function $f$ belongs to the {\it local Lipschitz space}  
$\operatorname{loc}\Lambda_{\omega,\infty}(\Omega)$ if  \eqref{rrt-1}  holds with a fixed positive constant $C$ whenever $z \in \Omega$ and  
$|z-w| < d_{\Omega}(z)/2$ (cf. \cite{Dy2,GM,L}), where $d_{\Omega}(z)$ denotes the Euclidean distance from $z$ to the boundary of $\Omega$.  
Moreover, $\Omega$ is called a $\Lambda_{\omega,\infty}$-extension domain if  
$$
\Lambda_{\omega,\infty}(\Omega) 
= \operatorname{loc}\Lambda_{\omega,\infty}(\Omega).
$$ 
For the geometric characterization of $\Lambda_{\omega,\infty}$-extension domains, we refer to \cite{GM}.  
In \cite{L}, Lappalainen generalized the characterization from \cite{GM} and proved that $\Omega$ is a $\Lambda_{\omega,\infty}$-extension domain if and only if every pair of points 
$z_1, z_2 \in \Omega$ can be joined by a rectifiable curve $\gamma \subset \Omega$ satisfying  
\begin{equation*}
\int_{\gamma} \frac{\omega(d_{\Omega}(\zeta))}{d_{\Omega}(\zeta)} \, ds(\zeta)  
\leq C \, \omega(|z_1 - z_2|)
\end{equation*} 
for some fixed positive constant $C=C(\omega,\Omega)$, 
where $ds$ denotes the arc length measure on $\gamma$.  
Furthermore, Lappalainen \cite[Theorem~4.12]{L} showed that 
$\Lambda_{\omega,\infty}$-extension domains exist only for fast majorants $\omega$. 
In particular,  $\mathbb{D}$ is a $\Lambda_{\omega,\infty}$-extension domain of $\mathbb{C}$ for 
any fast majorant $\omega$ (see \cite{Dy2}).

Before presenting our results, let us recall a classical Hardy-Littlewood theorem for analytic functions.

\begin{Thm}\label{Thm-A}{\rm (\cite[Theorem 41]{H-L})}
Let $f\in\mathscr{A}(\mathbb{D})\cap \mathcal{C}(\overline{\mathbb{D}})$.
If $f \in \Lambda_{\omega_{\beta}, \infty}(\mathbb{T})$, then  $f \in \Lambda_{\omega_{\beta}, \infty}( \mathbb{D} )$, where $\beta \in (0,1]$ and
 $\omega_{\beta}(t) = t^{\beta}$ for   $t\in[0,\infty)$.
\end{Thm}

Later,  Priwaloff \cite{Pri} proved the following result: 
 For a   harmonic function  $u:\mathbb{D}\rightarrow \mathbb{R}$
 and its harmonic conjugate $v$, if 
$u=\mathcal{P}[\varphi]$ with $\varphi\in \Lambda_{\omega_{\beta},\infty}(\mathbb{T})$, 
then $v=\mathcal{P}[\psi]$ with $\psi\in \Lambda_{\omega_{\beta},\infty}(\mathbb{T})$,
where $\beta\in (0,1)$ and $\omega_{\beta}(t)=t^{\beta}$ for $t\in[0,\infty)$. 
 
Combining  Priwaloff's result and Theorem  \ref{Thm-A}, we easily obtain the following result.


  \begin{Thm}\label{H-L} {\rm (\cite[Lemma C]{Pav-2008})}
 If $\varphi\in\Lambda_{\omega_{\beta},\infty}(\mathbb{T})$, then
$\mathcal{P}[\varphi]\in\Lambda_{\omega_{\beta},\infty}(\mathbb{D})$, where $\beta\in(0,1)$.
 \end{Thm}

In \cite{No}, Nolder and Oberlin generalized  Theorem \ref{H-L} in terms of differentiable majorants and
obtained the following  necessary and sufficient condition for  $\mathcal{P}[\varphi]\in\Lambda_{\omega, \infty}(\mathbb{D})$ under the assumption that $\varphi \in \Lambda_{\omega, \infty}(\mathbb{T})$.

\begin{Thm}{\rm (\cite[Lemma 3.2]{No})}\label{thm-B}
	Suppose that $\omega$ is a differentiable majorant   and that $\omega' $ is decreasing on $[0,\infty)$.
Then the following statements are equivalent:
\begin{enumerate}
  \item[$(a_{1})$]   If $\varphi \in \Lambda_{\omega, \infty}(\mathbb{T})$, then $\mathcal{P}[\varphi] \in \Lambda_{\omega, \infty}(\mathbb{D})$.
  \item[$(a_{2})$]  There exists a positive constant $C $ such that  for all $\delta \in (0, \pi]$,
				$$\delta \int_{\delta}^{\pi} \frac{\omega(t)}{t^{2}} d t \leq  C \omega(\delta).$$
\end{enumerate}
 \end{Thm}

In \cite{Dy1},  Dyakonov removed  the  differentiability  assumption on $\omega$ and established  the following  Hardy-Littlewood-type theorem.

 \begin{Thm}\label{Th-B}{\rm (\cite[Lemma 4]{Dy1})}  Let $\omega$ be a regular majorant.
  If   $\varphi\in\Lambda_{\omega,\infty}(\mathbb{T})$,
then $\mathcal{P}[\varphi]\in\Lambda_{\omega,\infty}(\overline{\mathbb{D}})$.
\end{Thm}

In 2025, based on Theorem \ref{Th-B}, Chen and Hamada obtained the following improved version of the Hardy-Littlewood-type theorem.

\begin{Thm}{\rm (\cite[Theorem 2.11]{Chen2025})}\label{thm-C}
	Let $\omega$ be a majorant and $p \in [1, \infty]$.
	\begin{enumerate}
		\item[$(b_{1})$] If $\varphi \in \Lambda_{\omega, p}(\mathbb{T})\cap \mathcal{C}(\mathbb{T})$, then $\mathcal{P}[\varphi] \in \Lambda_{\omega, p}(\overline{\mathbb{D}})$.
		\item[$(b_{2})$] There exists a positive constant $C$ such that for all $\delta \in (0, \pi] $,
		
		\be\label{ttp-1}\delta \int_{\delta}^{\pi} \frac{\omega(t)}{t^{2}} d t \leqslant C \omega(\delta).\ee
		
	\end{enumerate}
	Then, $(b_{2}) \Rightarrow (b_{1})$ for $p \in [1, \infty]$, and $(b_{1}) \Leftrightarrow (b_{2})$ for $p = \infty$.
\end{Thm}

Due to the limitations of their method, the authors of \cite{Chen2025} are unable to determine whether $(b_{1})$ implies $(b_{2})$ when $p \in [1,\infty)$.
 Very recently, Chen et al.~\cite{CCHL-2025} extended Theorem~\ref{thm-C} to higher-dimensional settings.  
However, for $p \in [1,\infty)$, it remains an open problem whether condition $(b_{1})$ implies $(b_{2})$.  
This is the so-called Hardy-Littlewood-type equivalence problem.

\begin{ques}{\rm (\cite[Question~2]{CCHL-2025})} (The Hardy-Littlewood-type equivalence problem)\label{Qes-2}
Does $(b_{1}) \Leftrightarrow (b_{2})$ hold for all $p \in [1,\infty]$ in Theorem \ref{thm-C}?
\end{ques}

The first purpose of this paper is to give an answer to this open problem. 

\begin{thm}\label{thm-1.1}
Let $\omega$ be a majorant and let $ p\in[1,\infty]$.
\begin{enumerate}
 \item[$(c_{1})$] If $p\in(1,\infty)$, then $(b_{1})$ holds for every
 majorant $\omega$.
\item[$(c_{2})$] If $p\in\{1,\infty\}$, then $(b_{1}) \Leftrightarrow (b_{2})$.
\end{enumerate}
\end{thm}

\begin{rem}\label{rem-1.1}
It follows from Theorem \ref{thm-1.1} that  the assumption (\ref{ttp-1}) is redundant   for  $p \in (1, \infty)$ in Theorem \ref{thm-C}.
Hence, $(b_{1})$ is not equivalent to $(b_{2})$ for $p \in (1, \infty)$.
\end{rem}

\smallskip

\subsection*{Applications of Theorem \ref{thm-1.1}}

Generally speaking, it is reasonable to expect that an analytic function behaves smoothly on the boundary when its derivative exhibits slow growth, and that the converse also holds in some sense. Nevertheless, Hardy and Littlewood established a principle that can be formulated in a remarkably precise manner, as stated in the following theorem.

\begin{Thm}\label{HL-20} {\rm (\cite[Theorem 5.1]{duren} or \cite[Theorem 40]{H-L})}
Let $f\in\mathscr{A}(\mathbb{D})\cap \mathcal{C}(\overline{\mathbb{D}})$.
Then for $\beta\in(0,1]$, 
we have $f\in\Lambda_{\omega_{\beta},\infty}(\mathbb{T})$ if and only if
$$
|f'(z)|=O\left(\frac{1}{(1-|z|)^{1-\beta}}\right)~~\mbox{as}~ |z|\rightarrow1^{-},
$$  where $\omega_{\beta}$ is defined in Theorem \ref{Thm-A}.
\end{Thm}

Recently, Das and Kaliraj \cite{DK} generalized Theorem \ref{HL-20} to complex-valued harmonic functions as follows
(see also \cite[Theorems 2.7 and 2.11]{Chen2025}).

\begin{Thm}\label{Thm-oj}{\rm (\cite[Theorem 2]{DK})}
Let  $f\in\mathscr{H}(\mathbb{D})\cap \mathcal{C}(\overline{\mathbb{D}})$.
Then  for $\beta\in(0,1)$, 
we have  $f\in\Lambda_{\omega_{\beta},\infty}(\mathbb{T})$ if and only if
$$\mathscr{M}_{f}(z)=O\left(\frac{1}{(1-|z|)^{1-\beta}}\right)~~\mbox{as}~ |z|\rightarrow1^{-}.
$$ 
\end{Thm}

In the following, we apply Theorem \ref{thm-1.1} to generalize Theorems \ref{HL-20} and \ref{Thm-oj} to a more general form.
 Before presenting our results, we recall some necessary definitions.

We say that $ f\in\mathscr{H}(\mathbb{D}) $ is a harmonic $ (K, K_{0}) $-quasiregular mapping 
if for every $ z \in \mathbb{D} $, 
$ J_{f}(z) \ge 0 $ and 
$$
\mathscr{M}_{f}^{2}(z) \le K \, J_{f}(z) + K_{0},
$$
where $ K \ge 1 $ and $ K_{0} \ge 0 $ are constants.
In particular, if $K_{0}\equiv0$, then a harmonic $(K,K_{0})$-quasiregular mapping is called a harmonic $K$-quasiregular mapping.
Let $\mathcal{S}_{h}(K,K_{0})$ denote the class of all harmonic $(K,K_{0})$-quasiregular mappings defined in  $\mathbb{D}$.  Clearly, $\mathscr{A}(\mathbb{D}) = \mathcal{S}_{h}(1,0)$.

The second author \cite[Theorem 2.3]{C-2026} generalized Theorems \ref{HL-20} and \ref{Thm-oj} 
for a weight function $\omega$ of Dini type which satisfies some assumption and for $f \in\mathcal{S}_{h}(K,K_{0})\cap \mathcal{C}(\overline{\mathbb{D}})$.
A concave increasing function $\omega$ on $[0,\infty)$ with $\omega(0)=0$ is called
a weight function of Dini type
if there exist constants $\delta_0>0$ and $C>0$ such that \eqref{eq2x} holds for $0<\delta<\delta_0$.
Since a concave increasing function on $[0,\infty)$ with $\omega(0)=0$ is a majorant (see \cite[Lemma 2.2]{A10})
and the proof of \cite[Theorem 2.3]{C-2026} is valid for a fast majorant $\omega$
which satisfies \eqref{N-O-condition},
it can be written as follows.

\begin{Thm}\label{Chen-2026}{\rm (cf. \cite[Theorem 2.3]{C-2026})}
 Let $\omega$ be a fast majorant satisfying the following condition:
 there is a positive constant $C$ such that for all $\delta \in (0,\pi]$,
		\begin{align}\label{N-O-condition}
		\delta \int_{\delta}^{\pi} \frac{\omega(t)}{t^2} \, dt \leq C\omega(\delta).
		\end{align}
Suppose that $f = u + iv\in\mathcal{S}_{h}(K,K_{0})\cap \mathcal{C}(\overline{\mathbb{D}})$, 
where $u$ and $v$ are real-valued functions, $K \geq 1$ and $K_{0} \geq 0$. 
Then, for $p \in [1,\infty]$,
$u \in \Lambda_{\omega,p}(\mathbb{T})$ if and only if	
		\begin{equation*}
			\begin{cases}
				\displaystyle  \left( \int_{0}^{2\pi} \left( \mathscr{M}_{f}\left(z e^{\mathrm{i}\theta}\right) \right)^p \, d\theta \right)^{\frac{1}{p}} \leq C
 \frac{\omega(d_{\mathbb{D}}(z))}{d_{\mathbb{D}}(z)} + C\sqrt{K_{0}}, & \text{if } p\in[1,\infty),\\ 
				\displaystyle \mathscr{M}_{f}(z) \leq C \frac{\omega(d_{\mathbb{D}}(z))}{d_{\mathbb{D}}(z)} + C\sqrt{K_{0}}, & \text{if } p = \infty,
			\end{cases}
		\end{equation*} where $C$ is a positive constant.
\end{Thm}

By using Theorem \ref{thm-1.1}, we 
can improve Theorems \ref{Thm-oj} and \ref{Chen-2026} into the following form.
We omit the proof of Theorem \ref{thm-1.3}, since it suffices to use an argument similar to that in the proof of \cite[Theorem 2.3]{C-2026}, with Theorem \ref{thm-1.1} in place of Theorem \ref{thm-C}.

\begin{thm}\label{thm-1.3}
 Let $\omega$ be a fast majorant. Suppose that $f = u + iv\in\mathcal{S}_{h}(K,K_{0})\cap \mathcal{C}(\overline{\mathbb{D}})$, where $u$ and $v$ are real-valued functions, $K \geq 1$ and $K_{0} \geq 0$. 
	\begin{enumerate}
		\item[($d_1$)]
		 Let $p \in (1,\infty)$. Then  $u \in \Lambda_{\omega,p}(\mathbb{T})$ if and only if
		\begin{equation*}
		 \left( \int_{0}^{2\pi} \left( \mathscr{M}_{f}\left(z e^{\mathrm{i}\theta}\right) \right)^p \, d\theta \right)^{\frac{1}{p}} \leq C
 \frac{\omega(d_{\mathbb{D}}(z))}{d_{\mathbb{D}}(z)} + C\sqrt{K_{0}}, 
	\end{equation*}  where $C$ is a positive constant.
 \item[($d_2$)]  Let $p \in \{1,\infty\}$. Assume that there is a positive constant $C$ such that for all $\delta \in (0,\pi]$,
		\[
		\delta \int_{\delta}^{\pi} \frac{\omega(t)}{t^2} \, dt \leq C\omega(\delta).
		\]
		Then $u \in \Lambda_{\omega,p}(\mathbb{T})$ if and only if	
		\begin{equation*}
			\begin{cases}
				\displaystyle  \int_{0}^{2\pi}  \mathscr{M}_{f}\left(z e^{\mathrm{i}\theta}\right)  \, d\theta  \leq C \frac{\omega(d_{\mathbb{D}}(z))}{d_{\mathbb{D}}(z)}
 + C\sqrt{K_{0}}, & \text{if } p=1,\\ 
				\displaystyle \mathscr{M}_{f}(z) \leq C \frac{\omega(d_{\mathbb{D}}(z))}{d_{\mathbb{D}}(z)} + C\sqrt{K_{0}}, & \text{if } p = \infty,
			\end{cases}
		\end{equation*} where $C$ is a positive constant.
	
	\end{enumerate}
\end{thm}
 
\begin{rem}\label{rem-1.10}
When $K=1$, $K_{0}=0$, $p=\infty$ and $\omega(t)=t^{\beta}~(t\geq0)$ with $\beta\in(0,1)$,
Theorem \ref{thm-1.3} reduces to Theorem  \ref{HL-20}.
\end{rem}

The following result is a Lipschitz characterization of  boundary functions of the moduli of analytic functions and their harmonic extensions.

 \begin{Thm}\label{Dya-2}{\rm(see \cite[Corollary 1 (ii)]{Dy1} \mbox{or} \cite[Theorem B]{P})}
Let $\omega$ be a regular majorant,  $f\in\mathscr{A}(\mathbb{D})\cap \mathcal{C}(\overline{\mathbb{D}})$, and let the boundary function of $|f|$ belong to $\Lambda_{\omega,\infty}(\mathbb{T})$.
Then $f$ is in $\Lambda_{\omega,\infty}(\mathbb{D})$ if and only if 
$$
\mathcal{P}[|f|](z)-|f(z)|\leq\, C\omega(d_{\mathbb{D}}(z))
$$
 for some positive constant $C$.
\end{Thm}



By using Theorem \ref{thm-C}, Chen and Hamada \cite{Chen2025}
improved and generalized  Theorem \ref{Dya-2} into the following form.
Note that the definition of $K$-quasiregular mappings in \cite{Chen2025} differs from that in this paper. In \cite{Chen2025}, 
$K$-quasiregular mappings are understood as locally univalent, sense-preserving $K$-quasiregular mappings. Thus, it can be stated as follows.

 \begin{Thm}{\rm (\cite[Theorem 2.13]{Chen2025})}\label{thm-F}
Suppose that $p\in[1,\infty]$, $f\in\mathscr{H}(\mathbb{D})\cap \mathcal{C}(\overline{\mathbb{D}})$ and $\omega$ is a regular majorant.
If the boundary function of $|f|$ is in $\Lambda_{\omega,p}(\mathbb{T})$ and $f$ is a locally univalent sense-preserving
$K$-quasiregular mapping in $\mathbb{D}$, where $K\in[1,\infty)$,
then the following statements are equivalent:

\begin{enumerate}
\item[{\rm ($e_{1}$)}] $f\in\Lambda_{\omega,p}(\mathbb{D})$;
\item[{\rm ($e_{2}$)}] There is a positive constant $C$ such that
 $$
 \left\{
 \begin{array}{ll}
 \left(\int_{0}^{2\pi}\left(\mathcal{P}[|f|](ze^{{\rm i}\eta})-|f(ze^{{\rm i}\eta})|\right)^{p}d\eta\right)^{\frac{1}{p}}
\leq C\omega(d_{\mathbb{D}}(z)), &p\in[1,\infty),
\\
\mathcal{P}[|f|](z)-|f(z)|\leq\, C\omega(d_{\mathbb{D}}(z)), &p=\infty.
\end{array}
\right.
$$
\end{enumerate}
\end{Thm}

In view of the proof of Theorem \ref{thm-F} in \cite{Chen2025}, 
the assumption on $f\in \mathscr{H}(\mathbb{D})$ that it is a locally univalent sense-preserving $K$-quasiregular mapping in $\mathbb{D}$
can be weakened to the following: 
for $z\in \mathbb{D}$ and $ \delta\in(0,1-|z|)$,
\begin{align}\label{Schwarz-Pick-quasiregular}
\mathscr{M}_f(z)\leq \frac{2K(M_{z,\delta}-|f(z)|)}{\delta},
\end{align}
where $M_{z,\delta}=\sup\{ |f(w)|: |w-z|<\delta\}$.

Since $f\in\mathscr{A}(\mathbb{D})$
satisfies \eqref{Schwarz-Pick-quasiregular} for $K=1$
by the Schwarz-Pick lemma (cf. \cite[Section 1.2]{ahl} and \cite[Lemma 3.3]{Chen2025}),
from Theorem \ref{thm-F} and the above argument, we obtain the following result immediately.

\begin{Core}{\rm (\cite[Corollary  2.15]{Chen2025})}\label{yy-cor}
Suppose that $p\in[1,\infty]$, $f\in\mathscr{A}(\mathbb{D})\cap \mathcal{C}(\overline{\mathbb{D}})$ and $\omega$ is a regular majorant.
If the boundary function of $|f|$ is in $\Lambda_{\omega,p}(\mathbb{T})$, then the following statements are equivalent:

\begin{enumerate}
\item[{\rm ($f_{1}$)}] $f\in\Lambda_{\omega,p}(\mathbb{D})$;
\item[{\rm ($f_{2}$)}] There is a positive constant $C$ such that
$$
 \left\{
 \begin{array}{ll}
 \left(\int_{0}^{2\pi}\left(\mathcal{P}[|f|](ze^{{\rm i}\eta})-|f(ze^{{\rm i}\eta})|\right)^{p}d\eta\right)^{\frac{1}{p}}
\leq C\omega(d_{\mathbb{D}}(z)), &p\in[1,\infty),
\\
\mathcal{P}[|f|](z)-|f(z)|\leq C\omega(d_{\mathbb{D}}(z)), &p=\infty.
\end{array}
\right.
$$
\end{enumerate}
\end{Core}




In the case $p\in (1,\infty)$,
by using Theorem \ref{thm-1.1}, we improve and generalize Theorem \ref{thm-F} into the following form.
We omit the proof of Theorem \ref{thm-1.5},
because it suffices to use an argument similar to that in the proof of \cite[Theorem 2.13]{Chen2025} and Theorem \ref{thm-1.1}
 instead of  Theorem \ref{thm-C}.

 \begin{thm}\label{thm-1.5}
Suppose that $p\in(1,\infty)$, $f\in\mathscr{H}(\mathbb{D})\cap \mathcal{C}(\overline{\mathbb{D}})$ and $\omega$ is a majorant.
Assume that the boundary function of $|f|$ is in $\Lambda_{\omega,p}(\mathbb{T})$.
We consider 
the following statements:
\begin{enumerate}
\item[{\rm ($g_{1}$)}] $f\in\Lambda_{\omega,p}(\mathbb{D})$;
\item[{\rm ($g_{2}$)}] There is a positive constant $C$ such that
 $$
 \left(\int_{0}^{2\pi}\left(\mathcal{P}[|f|](ze^{{\rm i}\eta})-|f(ze^{{\rm i}\eta})|\right)^{p}d\eta\right)^{\frac{1}{p}}
\leq C\omega(d_{\mathbb{D}}(z)).
$$
\end{enumerate}
Then, $(g_{1}) \Rightarrow (g_{2})$.
Conversely, if $\omega$ is fast and $f$
 satisfies \eqref{Schwarz-Pick-quasiregular}
for $z\in \mathbb{D}$ and $0<\delta<1-|z|$, where $K\in[1,\infty)$ is a constant,
then $(g_{2}) \Rightarrow (g_{1})$.

\end{thm}

From Theorem \ref{thm-1.5}, we obtain the following result immediately.

\begin{cor}\label{cor-1.6}
Suppose that $p\in(1,\infty)$, $f\in\mathscr{A}(\mathbb{D})\cap \mathcal{C}(\overline{\mathbb{D}})$ and $\omega$ is a majorant.
Assume that the boundary function of $|f|$ is in $\Lambda_{\omega,p}(\mathbb{T})$.
We consider 
the following statements.

\begin{enumerate}
\item[{\rm ($h_{1}$)}] $f\in\Lambda_{\omega,p}(\mathbb{D})$;
\item[{\rm ($h_{2}$)}] There is a positive constant $C$ such that
$$
 \left(\int_{0}^{2\pi}\left(\mathcal{P}[|f|](ze^{{\rm i}\eta})-|f(ze^{{\rm i}\eta})|\right)^{p}d\eta\right)^{\frac{1}{p}}
\leq C\omega(d_{\mathbb{D}}(z)).
$$
\end{enumerate}
Then, $(h_{1}) \Rightarrow (h_{2})$.
Conversely, if $\omega$ is fast,
then $(h_{2}) \Rightarrow (h_{1})$.
\end{cor}




Let $\phi$ be an analytic function of $\mathbb{D}$
into itself. We define
the composition operator $C_{\phi}$ on $\mathscr{H}(\mathbb{D})$ by $C_{\phi}(f)(z)=f(\phi(z))$ for $z\in\mathbb{D}$. For two sets  $X$,
$Y\subset \mathscr{H}(\mathbb{D})$, the set of all $\phi$ for which $C_{\phi}(X)\subset Y$ is
denoted by $\mathscr{F}(X,Y)$. 

In the following, for $x\in\mathbb{R}$, let $$\{x\}_{+}=\max\{x,0\}.$$
As another application of Theorem \ref{thm-1.1}, we obtain the following result.

\begin{thm}\label{thm-1.8}
Suppose that $p\in(1,\infty)$,   $\omega_{1}$ and $\omega_{2}$ are two fast majorants.
Let $\phi$ be an analytic function of $\mathbb{D}$ into itself.
Then the following conditions are equivalent.
\begin{enumerate}
\item[{\rm $(i_{1})$}] $\phi\in\mathscr{F}\big(\Lambda_{\omega_{1},\infty}(\mathbb{D})\cap\mathscr{H}(\mathbb{D}),\Lambda_{\omega_{2},p}(\mathbb{D})\cap\mathscr{H}(\mathbb{D})\big)$;

\item[{\rm $(i_{2})$}] There is a positive constant $C$ such that, for all $z\in\mathbb{D}$,
$$\left(\int_{0}^{2\pi}\left(|\phi'(ze^{{\rm i}\theta})|
\frac{\omega_{1}\big(d_{\mathbb{D}}(\phi(ze^{{\rm i}\theta}))\big)}{d_{\mathbb{D}}(\phi(ze^{{\rm i}\theta}))}\right)^{p}d\theta\right)^{\frac{1}{p}}
\leq C\frac{\omega_{2}\big(d_{\mathbb{D}}(z)\big)}{d_{\mathbb{D}}(z)};$$
\end{enumerate}
Furthermore, assume that
$\phi\in \mathcal{C}(\overline{\mathbb{D}})$ and
$\omega_{1}\in\mathscr{S}$, where $\mathscr{S}$ denotes the set consisting of those majorants $\omega_{1}$
for which $\omega_{1}$  are differentiable on $(0,1]$ and $\omega_{1}'$
  are also non-increasing on $(0,1]$ with $\sup_{t\in (0,1]}\frac{\omega_1(t)}{t\omega_1'(t)}<\infty$.
Then, the conditions {\rm $(i_{1})$} and {\rm $(i_{2})$} are equivalent to the conditions
 $(i_{3})$$\sim$$(i_{5})$.
\begin{enumerate}
\item[{\rm $(i_{3})$}]  $\omega_{1}(d_{\mathbb{D}}(\phi))\in\Lambda_{\omega_{2},p}(\overline{\mathbb{D}});$

\item[{\rm $(i_{4})$}] $\omega_{1}(d_{\mathbb{D}}(\phi))\in\Lambda_{\omega_{2},p}(\mathbb{T})$ and there is a positive constant $C$
such that, for all $r\in(0,1)$,
$$\left(\int_{0}^{2\pi}\left\{\omega_{1}(d_{\mathbb{D}}(\phi(re^{{\rm i}\theta})))-\omega_{1}(d_{\mathbb{D}}(\phi(e^{{\rm i}\theta})))\right\}_{+}^{p}d\theta\right)^{\frac{1}{p}}
\leq\,C\omega_{2}(d_{\mathbb{D}}(r));$$
\item[{\rm $(i_{5})$}]
$\omega_{1}(d_{\mathbb{D}}(\phi))\in\Lambda_{\omega_{2},p}(\mathbb{T})$ and there is a positive constant $C$
such that, for all $r\in(0,1)$,
$$\left(\int_{0}^{2\pi}\left(\omega_{1}(d_{\mathbb{D}}(\phi(re^{{\rm i}\theta})))-\mathcal{P}[\omega_{1}(d_{\mathbb{D}}(\phi))](re^{{\rm i}\theta})\right)^{p}d\theta\right)^{\frac{1}{p}}
\leq\,C\omega_{2}(d_{\mathbb{D}}(r)).$$
\end{enumerate}
\end{thm}

 In particular, if we take $\omega_{1}(t)=t^{\alpha}$ and $\omega_{2}(t)=t^{\beta}$  
 with $\alpha, \beta\in(0,1]$ and $t\in[0,\infty)$ in Theorem \ref{thm-1.8}, then
we obtain the following result. 

\begin{cor}\label{cor-1.3}
Suppose that   $p\in(1,\infty)$,  $\alpha, \beta\in(0,1]$ and
 $\phi$ is an analytic function of $\mathbb{D}$ into itself.
Then the following conditions are equivalent.
\begin{enumerate}
\item[{\rm $(j_{1})$}] $\phi\in\mathscr{F}\big(\Lambda_{\omega_{\alpha},\infty}(\mathbb{D})\cap\mathscr{H}(\mathbb{D}),\Lambda_{\omega_{\beta},p}(\mathbb{D})\cap\mathscr{H}(\mathbb{D})\big)$;

\item[{\rm $(j_{2})$}] There is a positive constant $C$ such that, for all $z\in\mathbb{D}$,
$$\left(\int_{0}^{2\pi}\left(|\phi'(ze^{{\rm i}\theta})|
{\big(d_{\mathbb{D}}(\phi(ze^{{\rm i}\theta}))\big)^{\alpha-1}}\right)^{p}d\theta\right)^{\frac{1}{p}}
\leq C\big(d_{\mathbb{D}}(z)\big)^{\beta-1};$$
\end{enumerate}
Furthermore, assume that
$\phi\in \mathcal{C}(\overline{\mathbb{D}})$.
Then, the conditions {\rm $(j_{1})$} and {\rm $(j_{2})$} are equivalent to the conditions
 $(j_{3})$$\sim$$(j_{5})$.
\begin{enumerate}
\item[{\rm $(j_{3})$}]  ${\big(d_{\mathbb{D}}(\phi)\big)^{\alpha}}\in\Lambda_{\omega_{\beta},p}(\overline{\mathbb{D}});$

\item[{\rm $(j_{4})$}] ${\big(d_{\mathbb{D}}(\phi)\big)^{\alpha}}\in\Lambda_{\omega_{\beta},p}(\mathbb{T})$ and there is a positive constant $C$
such that, for all $r\in(0,1)$,
$$\left(\int_{0}^{2\pi}\left\{(d_{\mathbb{D}}(\phi(re^{{\rm i}\theta})))^{\alpha}-
(d_{\mathbb{D}}(\phi(e^{{\rm i}\theta})))^{\alpha}\right\}_{+}^{p}d\theta\right)^{\frac{1}{p}}
\leq\,C(d_{\mathbb{D}}(r))^{\beta};$$
\item[{\rm $(j_{5})$}]
$(d_{\mathbb{D}}(\phi))^{\alpha}\in\Lambda_{\omega_{\beta},p}(\mathbb{T})$ and there is a positive constant $C$
such that, for all $r\in(0,1)$,
$$\left(\int_{0}^{2\pi}\left((d_{\mathbb{D}}(\phi(re^{{\rm i}\theta})))^{\alpha}-\mathcal{P}[(d_{\mathbb{D}}(\phi))^{\alpha}](re^{{\rm i}\theta})\right)^{p}d\theta\right)^{\frac{1}{p}}
\leq\,C(d_{\mathbb{D}}(r))^{\beta}.$$
\end{enumerate}
\end{cor}

\begin{rem}
The result of Theorem \ref{thm-1.8} shows that, for the $p\in(1,\infty)$ case of \cite[Theorem 2.17]{Chen2025}, the condition on 
$\omega_2$ can be relaxed from a regular majorant to a fast majorant.
We omit the proof of Theorem \ref{thm-1.8},
because it suffices to use an argument similar to that in the proof of \cite[Theorem 2.17]{Chen2025} and Theorem \ref{thm-1.1}
 instead of  Theorem \ref{thm-C}.
Also, the result of Corollary \ref{cor-1.3} shows that, for the $p\in(1,\infty)$ case of \cite[Corollary 2.19]{Chen2025}, the condition on $\beta$ can be relaxed from $\beta\in(0,1)$ to $\beta \in(0,1]$.
\end{rem}

The remainder of this paper is organized as follows. Section \ref{sec-2} contains the proof of Theorem \ref{thm-1.1}. 
For the proof, we use new functional analytical approaches as follows.
In the case $p\in (1,\infty)$,
we will show that a Hardy-Littlewood-type theorem holds for arbitrary majorant
by making use of Fourier series and the Hilbert transform.
In the case $p=1$,
we will prove the equivalence
by making use of Fourier series and the uniform boundedness principle.



\section{The proof of Theorem \ref{thm-1.1}}\label{sec-2}

 In order to prove Theorem \ref{thm-1.1}, we need to recall some basic definitions and
introduce some necessary terminologies. 

\subsection*{Sobolev spaces} 
 
For $p\in(0,\infty]$, we use $L^p(\mathbb{T})$  to denote the space of
all measurable  functions $f:\mathbb{T} \rightarrow \mathbb{C}$ satisfying
$\| f\| _{L^p(\mathbb{T})}<\infty$, where
\begin{equation*} 
\| f\| _{L^p(\mathbb{T})}=\begin{cases}
\displaystyle \left( \int_{0}^{2\pi}| f(e^{i\theta})| ^p\,d\theta\right)^{\frac{1}{p}}, \;\;\; &\text{if} \;\;p\in(0,\infty), \\
\displaystyle \operatorname*{esssup}_{\theta\in[0,2\pi]} \{| f(e^{i\theta})| \},\; \; &\text{if} \;\; p=\infty .
\end{cases}
\end{equation*}
 
For $p\in (1,\infty)$, a function $f:\mathbb T\rightarrow \mathbb{C}$ is said to belong to the Sobolev space
$W^{1,p}(\mathbb T)$,
if  $f\in L^{p}(\mathbb T) $  and there exists a function $g\in L^{p}(\mathbb T)$
such that
$$
f(e^{i\theta})=\int_{0}^{\theta} g(e^{i\eta})d\eta+C 
$$
with some constant $C $ (see \cite[Definition 1.11]{serov2017}). 
Obviously, $ g(e^{i\theta})=\frac{d f (e^{i\theta})}{d\theta}$ for almost every $\theta\in \mathbb{R}$.
For convenience, we write $ \frac{d f (e^{i\theta})}{d\theta}=\dot{f}(e^{i\theta})$ on $\mathbb{T}$.
Then $g(e^{i\theta})=\dot{f}(e^{i\theta})$ a.e. on $\mathbb{T}$.

\subsection*{Fourier series } 

For   $f\in L^1(\mathbb T)$ and $k\in\mathbb Z$, 
define the $k$-th Fourier
coefficient of $f$  by
\begin{equation}\label{ceq-2.1}
\widehat{f}(k)
=
\frac{1}{2\pi}
\int_{0}^{2\pi}
f(e^{i\theta})e^{-ik\theta}\,d\theta. 
\end{equation} 
Then the trigonometric  series
$
\sum_{k=-\infty}^{\infty}
\widehat f(k)e^{ik\theta} 
$
 is called the Fourier series of $f$, and we write
 $$
f(e^{i\theta})\sim \sum_{k=-\infty}^{\infty}
\widehat f(k)e^{ik\theta}  
 $$
   (see \cite{Duo2001} or \cite[p. 14]{serov2017}).

If $f\in W^{1,p}(\mathbb T)$ for some $p\in(1,\infty)$, 
it follows from \cite[Theorem 9.5 and Corollary 10.7]{serov2017}  that
\begin{align}\label{Fourier-equality}
f(e^{i\theta})=\sum_{k=-\infty}^{\infty}\widehat f(k)e^{ik\theta}
\end{align}
and the above Fourier series  
 converges absolutely and uniformly on $\mathbb{T}$.
Moreover, we see from the proof of \cite[Theorem 9.5]{serov2017} that
for any $k\in\mathbb Z$,
\begin{equation}\label{ceq-2.2}
\widehat{\dot{f}}( k)
=
ik\widehat f(k) .
\end{equation}

\subsection*{Hilbert transform} 
Given a function $f\in L^{p}(\mathbb{T})$ with $p\in[1,\infty]$,
the Hilbert transform of $f$, denoted by $\mathcal{H}[f]$,
 is defined  as follows:
\begin{equation}\label{ceq-z2.3}
 \mathcal{H}[f] (e^{i\theta})
=\frac{1}{2\pi}\text{P. V.} \int_{-\pi}^{\pi} f(e^{i\eta}) \cot\frac{\theta-\eta}{2}d\eta,
\end{equation} 
where $\theta\in \mathbb{R}$ and $\text{P. V.}$ stands for the principal value  (see \cite[Sections 3.10 and 3.14]{King}).
From \cite[p. 252]{Z59}, we know that the integral in \eqref{ceq-z2.3}  converges for almost all $\theta\in \mathbb{R}$.
If $f\in  L^{p}(\mathbb{T})$ with  $p\in(1,\infty]$, by \cite[Section 6.11]{King}, we know that
\begin{equation}\label{ceq-2.3}
 \mathcal{H}[f]  (e^{i\theta})\sim- i\sum_{k=-\infty}^{\infty}\text{sgn}(k)\widehat{f}(k) e^{ ik\theta} , 
\end{equation}  
where the sign function $\text{sgn}$ is defined by
 $$
\text{sgn}(x)
=
\begin{cases}
	1, & x>0,\\
	0, & x=0,\\
	-1, & x<0.
\end{cases}
$$
If  $p\in(1,\infty)$, 
 we deduce from \cite[Equation (6.167)]{King} that there exists a positive 
constant $C_1=C_1(p)$  such that
\begin{equation*}
\|\mathcal{H}[f]\|_{L^p(\mathbb T)}
\leq
C_1\|f\|_{L^p(\mathbb T)},
\end{equation*}
Therefore, for any $f\in W^{1,p}(\mathbb{T})$ with $p\in(1,\infty)$,  we have 
\begin{equation}\label{ceq-2.4}
	\| \mathcal{H}[\dot{f}]\|_{L^p( \mathbb{T})}
	\leq
	C_1\| \dot{f}\|_{L^p( \mathbb{T})} 
\end{equation}
and
\begin{equation}\label{ceq-2.5}
 \mathcal{H}[\dot{f}]  (e^{i\theta})\sim -i\sum_{k=-\infty}^{\infty}\text{sgn}(k)\widehat{\dot{f}}(k) e^{ ik\theta}
=  \sum_{k=-\infty}^{\infty} |k|\widehat f( k) e^{ ik\theta},
\end{equation}
where  \eqref{ceq-2.2}  and \eqref{ceq-2.3} are applied in \eqref{ceq-2.5}.

\subsection*{Proof of Theorem \ref{thm-1.1}}
 We divide the proof of this theorem into two cases.

\medskip
\noindent $\mathbf{Case~1}.$ Let $p\in(1,\infty)$.
\medskip

Before proving this case, we first recall or establish several necessary lemmas. 

\begin{Lem}\label{lem-D}\rm{(}\cite[p. 19]{Z59}\rm{)}
	For  $p\in[1,\infty)$,  Minkowski's integral inequality 
is
	$$
	\left(\int_{A_{1}}\left|\int_{B_{1}} \mathcal{X}(\zeta, \xi) d \mu_{\xi}\right|^{p} d \mu_{\zeta}\right)^{\frac{1}{p}} \leqslant \int_{B_{1}}\left(\int_{A_{1}}|\mathcal{X}(\zeta, \xi)|^{p} d \mu_{\zeta}\right)^{\frac{1}{p}} d \mu_{\xi},
	$$
	where
 $A_{1}$  and  $B_{1}$  are two measurable sets with positive measures $ d \mu_{\zeta} $ and $ d \mu_{\xi}$, respectively,
  and $\mathcal{X}$  is integrable on $ A_{1} \times B_{1}$.
\end{Lem}

For any $r\in[0,1)$ and $\varphi \in L^{1}(\mathbb{T})$, we define
\begin{equation}\label{ceq-3.1}
\mathcal{P}_r[\varphi](e^{i\theta })
=\frac{1}{2\pi}
\int_{0}^{2\pi}\mathcal{P}_r(\eta)\varphi\big(e^{i(\theta-\eta)}\big)\,d\eta
\end{equation}
on $\mathbb{T}$, where
$$
\mathcal{P}_{r}(\eta)=\frac{1-r^{2}}{1+r^{2}-2r\cos\eta} 
$$
for   $\eta\in \mathbb{R}$.
Then  
\begin{equation*} 
\mathcal{P}_r[\varphi](e^{i\theta })
=\mathcal{P}[\varphi](re^{i\theta }),
\end{equation*}
and,  
it follows directly from \cite[p. 71]{KY2004} or \cite[p. 24]{serov2017} that
\begin{align}\label{P_r[f]-Fourier}
 \mathcal{P}_{r}[\varphi](e^{i\theta})=\sum_{k=-\infty}^{\infty}\widehat{\varphi}(k)r^{|k|}e^{ik\theta}.
\end{align}
The operator $\mathcal{P}_{r}$ possesses the following contraction property,
which will be used later on.

\begin{Lem}\label{lem-E} \rm{(}\cite[Theorem 1]{Be1992} or \cite[Equation (17)]{Riesz1928}\rm{)}
The operator $\mathcal{P}_{r}$ is a contraction mapping from $L^{p}(\mathbb{T})$ to $L^{p}(\mathbb{T})$,
where $r\in[0,1)$ and $p\in[1,\infty]$,	
 namely
	$$
	\|\mathcal{P}_{r}[\varphi]\|_{L^{p}(\mathbb{T})}\leq \|\varphi\|
	_{L^{p}(\mathbb{T})}.
$$
\end{Lem}

\begin{lem}\label{clem-3.1}
Let $p\in[1,\infty]$ and $\omega$ be a  majorant. Suppose that there exists a  positive constant
	$C_{2}$ such that for every $f\in \mathcal{C}(\mathbb{T})\cap\Lambda_{\omega,p}(\mathbb{T})$
	and every  $r\in(0,1)$,
	\begin{equation}\label{ceq-3.2}
		\left\|f-\mathcal{P}_{r}[f]\right\|_{L^{p}(\mathbb{T})}
		\leq
		C_{2} \|f \|_{\Lambda_{\omega,p}(\mathbb{T}),s}
		\omega(1-r).
	\end{equation}
	Then there exists a  positive  constant $ C_{3} $ such that	for all $z,w\in\overline{\mathbb{D}}$,
	\begin{equation*}
		\mathcal L_{p}[\mathcal{P}[f]](z,w)
		\leq
		C_{3}\|f\|_{\Lambda_{\omega,p}(\mathbb T),s}
		\omega(|z-w|),
	\end{equation*}
where $ C_{3} $ is independent of $f$.

\end{lem}
\begin{proof}
 Following Step 1 in the proof of \cite[Theorem 2.11]{Chen2025}, 
 we replace the use of \cite[Lemma 3.8]{Chen2025} with \eqref{ceq-3.2}. 
 The desired result then follows immediately, and we therefore omit the routine details. 
\end{proof}

In view of  Lemma \ref{clem-3.1}, to prove Theorem \ref{thm-1.1} for the case $p\in(1,\infty)$, it is sufficient to 
establish the inequality \eqref{ceq-3.2}  for all $f\in \mathcal{C}(\mathbb{T})\cap\Lambda_{\omega,p}(\mathbb{T})$
	and   $r\in(0,1)$,
namely, Lemma \ref{clem-3.3} below.
For this, we need the following auxiliary result.

\begin{lem}\label{clem-3.2}
Let $p\in(1,\infty)$.
Then there exists a  positive constant
	$C_{4}=C_{4}(p)$ such that for all  $r\in[1/2,1)$ and $f\in W^{1,p}(\mathbb{T})$,
	\begin{equation*} 
		\left\|f-\mathcal{P}_{r}[f]\right\|_{L^{p}(\mathbb{T})}
		\leq
		C_{4}(1-r)\|\dot{f}\|_{L^{p}(\mathbb{T})}.
	\end{equation*}
	\end{lem}
\begin{proof}
For any $\theta\in \mathbb{R}$, $r\in[1/2,1)$ and $f\in W^{1,p}(\mathbb{T})$ with $p\in(1,\infty)$,
combining \eqref{Fourier-equality} and \eqref{P_r[f]-Fourier}, we obtain
\begin{equation*}
 f (e^{i\theta})- \mathcal{P}_{r}[f](e^{i\theta})
 =\sum_{k=-\infty}^{\infty}\widehat f(k)( 1-r^{|k|}) e^{ik\theta}
 =\sum_{k=-\infty}^{\infty}\int_r^1  |k|\widehat f(k)\rho^{|k|-1} e^{ik\theta} \,d\rho.
\end{equation*}
Since the absolute convergence of the Fourier series for functions in  $W^{1,p}(\mathbb{T})$ with $p\in(1,\infty)$ (\cite[Theorem 9.5]{serov2017}) guarantees that
\begin{align*}
\sum_{k=-\infty}^{\infty}\int_r^1 | k \widehat{f}(k)
\rho^{|k|-1}e^{ik\theta}|\,d\rho =
	\sum_{k=-\infty}^{\infty}|\widehat{f}(k)|
	\bigl(1-r^{|k|}\bigr) \leq
	\sum_{k=-\infty}^{\infty}|\widehat{f}(k)|<\infty,
\end{align*}
we deduce from \cite[p. 55]{folland}  that
 \begin{equation*}
\sum_{k=-\infty}^{\infty}\int_{r}^{1} |k| \widehat{f}(k) \rho^{|k|-1}e^{ik\theta} \,d\rho
  =\int_r^1\sum_{k=-\infty}^{\infty}|k|\widehat{f}(k)
\rho^{|k|-1}e^{ik\theta} \,d\rho.
\end{equation*}
Therefore,
 \begin{equation*}
 f (e^{i\theta})- \mathcal{P}_{r}[f](e^{i\theta})
=\int_r^1 \sum_{k=-\infty}^{\infty} |k| \widehat{f}(k)
\rho^{|k|-1}e^{ik\theta} \,d\rho.
 \end{equation*} 
 Further, by \eqref{ceq-2.5} and \eqref{P_r[f]-Fourier}, we obtain that
 \begin{align*}
 \mathcal{P}_\rho[\mathcal{H}[ \dot{f}]](e^{i\theta})
 &=
 \sum_{k=-\infty}^{\infty} |k| \widehat{f}(k)
\rho^{|k|}e^{ik\theta},
 \end{align*}
which implies that
  \begin{equation*}
    f (e^{i\theta})- \mathcal{P}_{r}[f](e^{i\theta})
    = \int_{r}^{1}\mathcal{P}_\rho[\mathcal{H}[ \dot{f}]](e^{i\theta})\,\frac{  d\rho}{\rho}.
\end{equation*} 
This, together with  Lemma \ref{lem-D},
means that
 \begin{align}\label{ceq-3.3}
 \left\|f-\mathcal{P}_r[f] \right\|_{L^p(\mathbb{T})}
 &=
 \left(\int_{0}^{2\pi} \left|\int_{r}^{1}\mathcal{P}_\rho[\mathcal{H}[ \dot{f}]] (e^{i\theta}) \frac{d\rho}{\rho}\right|^{p}
 d\theta\right)^{\frac{1}{p}}\\\nonumber
 &\leq
 \int_{r}^{1} \left\|\mathcal{P}_{\rho}[\mathcal{H}[ \dot{f}]  \right \|_{L^{p}(\mathbb{T})}\frac{d\rho }{\rho} .
 \end{align}
 
By \eqref{ceq-2.4} and Lemma \ref{lem-E}, we see that  
\begin{align*}
 \int_{r}^{1} \left\|\mathcal{P}_\rho[\mathcal{H}[ \dot{f}]] \right  \|_{L^{p}(\mathbb{T})}\frac{d\rho }{\rho}
 \leq
 \left \|\mathcal{H}[ \dot{f}]   \right \|_{L^p(\mathbb{T})} \log\frac{1}{r}
 \leq
 C_{1}\left\|\dot{f} \right\|_{L^p(\mathbb{T})}\log\frac{1}{r},
\end{align*}
where  $C_{1}=C_{1}(p)$ is the constant from \eqref{ceq-2.4}. 
  Notice that for any $r\in[1/2,1)$,
  $$
 \log\frac{1}{r}
 \leq
 \frac{1-r}{r}
 \leq
 2(1-r).
$$
Then we deduce from \eqref{ceq-3.3} and the above inequalities that
\begin{equation*}
 \|f-\mathcal{P}_r[f]\|_{L^p(\mathbb{T})}
 \leq
 2 (1-r) C_1\|\dot{f}\|_{L^p(\mathbb{T})}.
\end{equation*}
By letting $C_{4}=2C_{1}$, we see that the lemma is true.
 \end{proof}

Now, by using Lemma \ref{clem-3.2},
we are ready to show that the inequality  \eqref{ceq-3.2} holds true for all $f\in \mathcal{C}(\mathbb{T})\cap\Lambda_{\omega,p}(\mathbb{T})$
	and   $r\in(0,1)$.

\begin{lem}\label{clem-3.3}
Let $p\in(1,\infty)$ and $\omega$ be a  majorant. Then there exists a positive constant $C_{5}=C_{5}(p)$ such  that  for every
$f\in \mathcal{C}(\mathbb{T})\cap\Lambda_{\omega,p}(\mathbb{T})$ and every $r\in(0,1)$,
\begin{equation}\label{ceq-3.4}
 \|f-\mathcal{P}_r[f]\|_{L^p(\mathbb{T})}
 \leq
 C_{5}\|f\|_{\Lambda_{\omega,p}(\mathbb{T}),s}\omega(1-r).
\end{equation}

\end{lem}

\begin{proof}
In order to prove the lemma, first we assume that  $r\in[1/2,1)$. 
Put  $t=1-r$. Then $t \in(0,1/2]$.
 For any $t \in(0,1/2]$,   $\theta\in \mathbb{R}$ and $f\in \mathcal{C}(\mathbb{T})\cap\Lambda_{\omega,p}(\mathbb{T})$ with $p\in(1,\infty)$,
define
\begin{equation}\label{ceq-3.5}
 \mathcal{S}_{t}[f](e^{i\theta})
 =:
 \frac{1}{t}\int_{0}^{t}f(e^{i(\theta+\eta)}) \,d\eta
 =
 \frac{1}{t}\int_{\theta}^{\theta+t}f(e^{i\eta})\,d\eta.
\end{equation}
Since $f\in \mathcal{C}(\mathbb{T})$, we easily obtain that $\mathcal{S}_{t}[f] \in W^{1,p}(\mathbb{T})$
and
$$
\frac{d}{d\theta} \mathcal{S}_{t}[f](e^{i\theta})
 =
 \frac{f(e^{i(\theta+t)}) -f(e^{i\theta})}{t}.
$$
For convenience,   we write
$$  (\mathcal{S}_{t}[f])'(e^{i\theta})=\frac{d}{d\theta} \mathcal{S}_{t}f(e^{i\theta}).$$
Then we have
 \begin{equation*} 
 t \left\|(\mathcal{S}_{t}[f])' \right\|_{L^p(\mathbb{T})}
 =\left(\int_0^{2\pi}\left|  f(e^{i(\theta+t)})  - f(e^{i \theta }) \right|^{p}\,d\theta\right)^{\frac{1}{p}}.
\end{equation*}
By combining this equality and  \eqref{ceq-1.3}, we find that for any $p\in(1,\infty)$,
\begin{equation}\label{ceq-3.6}
 t  \left\|(\mathcal{S}_{t}[f])' \right \|_{L^p(\mathbb{T})}
 =   \left\|\Delta_{t}f   \right\|_{L^{p}(\mathbb{T})}
 \leq
 \sup_{|h|\leq t}\|\Delta_{h} f\|_{L^p(\mathbb{T})}.
\end{equation}

Recall that
 $$\mathcal{P}_{r}[\mathcal{S}_{t}[f]-f]=\mathcal{P}_{r}[\mathcal{S}_{t}[f]]-\mathcal{P}_{r}[f].$$
 Then we deduce from Minkowski's inequality 
 that
\begin{align}\label{ceq-3.7}
 \left\|f-\mathcal{P}_{r}[f] \right \|_{L^p(\mathbb{T})}
\leq&
  \left\|f-\mathcal{S}_{t}[f]   \right  \|_{L^p(\mathbb{T})}
 + \left\|\mathcal{S}_{t}[f]-\mathcal{P}_{r}[\mathcal{S}_{t}[f]] \right  \|_{L^p(\mathbb{T})}\\ \nonumber
& + \left\|\mathcal{P}_{r}[\mathcal{S}_{t}[f]-f] \right  \|_{L^p(\T)}\\ \nonumber
 \leq&
 2 \left\|f-\mathcal{S}_{t}[f]   \right  \|_{L^p(\mathbb{T})}
 +t C_{4}  \left\|(\mathcal{S}_{t}[f])'    \right\|_{L^p(\mathbb{T})},~\text{(by Lemmas \ref{lem-E} and  \ref{clem-3.2})}
\end{align}
 where $C_{4}=C_{4}(p)$ is the constant from Lemma \ref{clem-3.2} and
\begin{align}\label{ceq-3.8}
\quad \|f-\mathcal{S}_{t}[f]\|_{L^p(\mathbb{T})}
 &=\left(\int_{0}^{2\pi}\left|\frac{1}{t}\int_0^{t}\bigl(f( e^{i\theta})-f(e^{i(\theta+\eta)} ) \bigr)\,d\eta\right|^{p}d\theta\right)^{\frac{1}{p}} \qquad\quad \;\text{(by   \eqref{ceq-3.5})} \\ \nonumber
 &\leq\frac{1}{t}\int_{0}^{t}\left(\int_0^{2\pi}\left|f( e^{i\theta})-f(e^{i(\theta+\eta)} ) 
 \right|^{p}\,d\theta\right)^{\frac{1}{p}}d\eta \quad \;\quad\text{(by  Lemma \ref{lem-D})}   \\ \nonumber
 &\leq
 \sup_{|h|\leq t}\|\Delta_{h} f\|_{L^p(\mathbb{T})}.\qquad\qquad\qquad\;\qquad\quad\quad\quad\quad\qquad\text{(by  \eqref{ceq-1.3})}
\end{align}
Substitute \eqref{ceq-3.6} and \eqref{ceq-3.8} into \eqref{ceq-3.7}. 
Then we arrive at the following inequality:
\begin{equation}\label{ceq-3.9}
 \|f-\mathcal{P}_r[f]\|_{L^p(\mathbb{T})}
 \leq
(2+C_{4}) \sup_{|h|\leq t}\|\Delta_{h} f\|_{L^p(\mathbb{T})}.
\end{equation}

Since   $|e^{ih}-1|\leq|h|$ for any $h\in \mathbb{R}$,
it follows directly from
\eqref{ceq-1.4} and the increasing property of $\omega$ that
\begin{equation*}
\sup_{|h|\leq t}\|\Delta_h f\|_{L^p(\mathbb{T})}
 \leq
 \|f\|_{\Lambda_{\omega,p}(\mathbb{T}),s} \sup_{|h|\leq t}\omega(|e^{ih}-1|)
  \leq
 \|f\|_{\Lambda_{\omega,p}(\mathbb{T}),s}\omega(t).
\end{equation*}
This, together with \eqref{ceq-3.9}, yields that    for any $r\in[1/2,1)$,
\begin{equation}\label{ceq-3.10}
 \|f-\mathcal{P}_r[f]\|_{L^p(\mathbb{T})}
 \leq
(2+C_{4})  \|f\|_{\Lambda_{\omega,p}(\mathbb{T}),s}\omega(1-r).
\end{equation}
\smallskip

Second, we consider the case $r\in(0,1/2)$. 
For   any
$f\in \mathcal{C}(\mathbb{T})\cap\Lambda_{\omega,p}(\mathbb{T})$ with $p\in(1,\infty)$, let
$$
f_{\mathbb{T}}=:\frac{1}{2\pi}\int_{-\pi}^{\pi}f(e^{i\theta})\,d\theta.
$$
Notice that \eqref{ceq-1.5} guarantees that 
$$f_{\mathbb{T}}-\mathcal{P}_r[f](e^{i\theta})=\mathcal{P}_r[f_{\mathbb{T}}-f](e^{i\theta}),$$
where $r\in(0,1/2)$ and $\theta\in \mathbb{R}$.
Then we deduce from Minkowski's inequality and Lemma \ref{lem-E}  that
\begin{equation}\label{ceq-3.11}
 \|f-\mathcal{P}_r[f]\|_{L^p(\mathbb{T})}
  \leq  \|f-f_{\mathbb{T}}\|_{L^p(\mathbb{T})}+ \|f_{\mathbb{T}}-\mathcal{P}_r[f]\|_{L^p(\mathbb{T})}
 \leq
 2\|f-f_{\mathbb{T}}\|_{L^p(\mathbb{T})} .
\end{equation}

By simple calculations, we know that
$$
f_{\mathbb{T}}= \frac{1}{2\pi}\int_{-\pi}^{\pi}f(e^{i\theta})\,d\theta
=\frac{1}{2\pi}\int_{-\pi}^{\pi}f(e^{i(\eta+h)})\,dh,
$$
where $\eta\in \mathbb{R}$.
Then
\begin{align}\label{ceq-3.12}
 \|f-f_{\mathbb{T}}\|_{L^p(\mathbb{T})}
 &=\left(\int_{-\pi}^{\pi}\left|\frac{1}{2\pi}\int_{-\pi}^{\pi}
\left(f(e^{i\eta})-f(e^{i(\eta+h)})\right) dh \right|^{p}d\eta\right)^{\frac{1}{p}}\\\nonumber
 &\leq
\frac{1}{2\pi}\int_{-\pi}^{\pi}
\left(\int_{-\pi}^{\pi}
\left|f(e^{i\eta})-f(e^{i(\eta+h)})\right|^{p}\,d\eta\right)^{\frac{1}{p}} dh  \;\quad\text{(by Lemma \ref{lem-D})}\\\nonumber
  &=
 \frac{1}{2\pi}\int_{-\pi}^{\pi}
 \|\Delta_h f\|_{L^p(\mathbb{T})}\,dh . \quad   \quad \quad \quad\quad \quad  \quad \quad \quad \quad\quad\text{(by \eqref{ceq-1.3})}
\end{align}
For any $h\in [-\pi,\pi]$, \eqref{ceq-1.4} guarantees that
\begin{align}\label{ceq-3.13}
 \|\Delta_h f\|_{L^p(\mathbb{T})}\leq \|f\|_{\Lambda_{\omega,p}(\mathbb{T}),s}\omega(|e^{ih}-1|)\leq \|f\|_{\Lambda_{\omega,p}(\mathbb{T}),s}\omega(2).
\end{align}
Then the combination of \eqref{ceq-3.12} and \eqref{ceq-3.13} gives that
\begin{equation}\label{ceq-3.14}
 \|f-f_{\mathbb{T}}\|_{L^p(\mathbb{T})}
 \leq\|f\|_{\Lambda_{\omega,p}(\mathbb{T}),s}\omega(2).
\end{equation}

By substituting \eqref{ceq-3.14} into \eqref{ceq-3.11}, we find that
$$
 \left\|f-\mathcal{P}_{r}[f]\right\|_{L^p(\mathbb{T})}
 \leq
 2\|f\|_{\Lambda_{\omega,p}(\mathbb{T}),s}\omega(2).
$$
Because $1-r\in(1/2,1)$, $\omega(x)/x$ is non-increasing in $(0,\infty)$ and
$\omega$ is increasing in $[0,\infty)$, we obtain that
$$
 \omega(2)\leq4\omega(1/2)\leq4\omega(1-r).
$$
Hence,  for any $r\in(0,1/2)$,
\begin{equation}\label{ceq-3.15}
\left \|f-\mathcal{P}_r[f]\right\|_{L^p(\mathbb{T})}
 \leq
 8\|f\|_{\Lambda_{\omega,p}(\mathbb{T}),s}\omega(1-r).
\end{equation}

Now, we conclude from \eqref{ceq-3.10} and \eqref{ceq-3.15} that  \eqref{ceq-3.4} 
holds  for all $r\in(0,1)$ with $C_{5}=8+C_{4}$.
\end{proof}

By  \eqref{ceq-1.2}, Lemmas \ref{clem-3.1} and \ref{clem-3.3}, we see that if $\varphi \in \mathcal{C}(\mathbb{T})\cap \Lambda_{\omega, p}(\mathbb{T})$ with $p \in (1, \infty)$, then 
$\mathcal{P}[\varphi]\in \Lambda_{\omega,p}(\overline{\mathbb{D}})$
and 
	\begin{equation*}
\left \|\mathcal{P}[\varphi] \right\|_{\Lambda_{\omega,p} (\overline{\mathbb{D}}),s} \leq C_{3} \|\varphi\|_{\Lambda_{\omega,p}(\mathbb{T}),s},
\end{equation*}
 where $ C_{3} $ is the constant from Lemma \ref{clem-3.1}.
Hence, $(b_{1})$ holds  for every majorant $\omega$.
This completes the proof of Theorem \ref{thm-1.1} for $p\in(1,\infty)$. 

\medskip

\noindent $\mathbf{Case~2}.$ Let $p\in\{1,\infty\}$.
\medskip

We only need to prove the case $p=1$, as the case $p=\infty$ follows directly from Theorem \ref{thm-C}.
We divide the proof of this case into several lemmas, as detailed below. 

 First, we show that  for every majorant $\omega$, the space $X_\omega(\mathbb{T})=: \mathcal{C}(\mathbb{T})\cap\Lambda_{\omega,1}(\mathbb{T})$ becomes a Banach space
when endowed with the norm
\begin{equation}\label{ceq-4.1}
	\|f\|_{X_\omega(\mathbb{T})}
	= \|f\|_{ \max}
	+
	\|f\|_{\Lambda_{\omega,1}(\mathbb{T}),s} ,
\end{equation}
where $ \|f\|_{ \max}=\max_{\theta\in \mathbb{R}} |f(e^{i\theta})|$.
 
  \begin{lem}\label{clem-4.1}
For every majorant $\omega$, $X_\omega(\mathbb{T})$ is a Banach space.
\end{lem}

\begin{proof}
 Let $\{f_n\}$ be a Cauchy sequence in $X_\omega(\mathbb{T})$.
For every $m,n\in \mathbb{N}$, it follows from  
$$
	\|f_n-f_m\|_{\max}
	\leq
	\|f_n-f_m\|_{X_\omega(\mathbb{T})} 
$$
  that   $\{f_n\}$ is also a Cauchy   sequence in  $\mathcal{C}(\mathbb{T})$.
Then there exists a mapping $f\in \mathcal{C}(\mathbb{T})$ such that
$\{f_n\}$ converges to $f$ in $\mathcal{C}(\mathbb{T})$,
i.e.,
\begin{equation}\label{ceq-4.2}
\lim_{n\rightarrow\infty} \|f_n-f\|_{\max}=0.
\end{equation}

Let $\varepsilon>0$.
Because $\{f_n\}$ is a Cauchy sequence in $X_\omega(\mathbb{T})$,
  there exists a positive integer $n_{\varepsilon}$ such that 
$$
	\|f_n-f_m\|_{\Lambda_{\omega,1}(\mathbb{T}),s}
	\leq\|f_n-f_m\|_{X_\omega(\mathbb{T}) }
	<\varepsilon.
$$
whenever $n,m\geq n_{\varepsilon}$.
Thus, for every $h\in[-\pi,\pi]$, we deduce from \eqref{ceq-1.4} that
\begin{equation}\label{ceq-4.3}
	\|\Delta_h(f_n-f_m)\|_{L^1(\mathbb{T})}
	\leq
\| f_n-f_m \|_{\Lambda_{\omega,1}(\mathbb{T}),s}\omega(|e^{ih}-1|)
	\leq \varepsilon\omega(|e^{ih}-1|).
\end{equation}
 Since \eqref{ceq-4.2} implies that $\{f_n\}$ converges uniformly to $f$
 on $\mathbb{T}$,
 we see that for each fixed $h\in[-\pi,\pi]$ and $m\in \mathbb{N}$,
 the  sequence  $\{ \Delta_h(f_n-f_m) \}_{n=1}^{\infty}$ also converges uniformly to
  a continuous function $\Delta_h(f-f_m)$
 on $\mathbb{T}$.
By letting $n\rightarrow\infty$ in \eqref{ceq-4.3},
we obtain that for every $h\in[-\pi,\pi]$ and $m\geq n_{\varepsilon}$,
\begin{equation*}
	\|\Delta_h(f-f_m)\|_{L^1(\mathbb{T})}
	\leq
	\varepsilon\omega(|e^{ih}-1|).
\end{equation*}
Combining the above inequality with \eqref{ceq-1.4}  yields that for every $\varepsilon>0$ and $m \geq n_{\varepsilon}$,
\begin{equation}\label{ceq-4.4}
	\|f-f_m\|_{\Lambda_{\omega,1}(\mathbb{T}),s}
	\leq\varepsilon.
\end{equation}
Therefore,
\begin{equation*}
	\| f\|_{\Lambda_{\omega,1}(\mathbb{T}),s}
\leq 	\|f-f_m\|_{\Lambda_{\omega,1}(\mathbb{T}),s}+	\|f_m\|_{\Lambda_{\omega,1}(\mathbb{T}),s}
\leq \varepsilon+\|f_m\|_{\Lambda_{\omega,1}(\mathbb{T}),s}  
\end{equation*}
whenever $m\geq n_{\varepsilon}$.
Recall that $f\in \mathcal{C}(\mathbb{T})$.
Then $f\in X_\omega(\mathbb{T})$.
This, together with  \eqref{ceq-4.2} and \eqref{ceq-4.4},
implies  that $\{f_n\}$ converges to $f$ in $X_\omega(\mathbb{T})$, and so, $ X_\omega(\mathbb{T})$  is a Banach space.
\end{proof}

 \begin{lem}\label{clem-4.2}
Suppose that $\omega$ is a majorant and  $(b_1)$ holds     for $p=1$. 
Then there exists a positive constant $C_{6}=C_{6}(\omega) $ such that
	\begin{equation}\label{ceq-4.5}
		\left\|f-\mathcal{P}_{1-r}[f] \right \|_{L^1(\mathbb{T})}
		\leq
		C_{6}\omega(r)
		\|f\|_{X_\omega(\mathbb{T})}.
	\end{equation}
for every $f\in X_{\omega}(\mathbb{T})$ and every $r\in(0,1/2)$.
\end{lem}

\begin{proof}
In order to prove \eqref{ceq-4.5}, we define
	\begin{equation}\label{ceq-4.6}
	\mathcal{T}_r[f](e^{i\theta})
	=
	\frac{f(e^{i\theta})-\mathcal{P}_{1-r}[f](e^{i\theta})}{\omega(r)} 
	\end{equation}
for every  $r\in(0,1/2)$, $\theta\in \mathbb{R}$ and $f\in X_{\omega}(\mathbb{T})$. 
Since Lemma \ref{lem-E} guarantees that
$$
 \left \|\mathcal{T}_r[f] \right\|_{L^1(\mathbb{T})}
\leq \frac{1}{\omega(r)}\left(\| f\|_{L^1(\mathbb{T})} + \left \|\mathcal{P}_{1-r}[f] \right\|_{L^1(\mathbb{T})}\right)
\leq \frac{2}{\omega(r)}\| f\|_{L^1(\mathbb{T})}
\leq
		\frac{4\pi}{\omega(r)}\|f\|_{X_\omega(\mathbb{T})},
$$	
we know that
$\mathcal{T}_r $ is a bounded  linear operator from $X_\omega(\mathbb{T})$ to $L^1(\mathbb{T})$.
Then we see from  the uniform boundedness principle \cite[Theorem 2.2]{bre} and Lemma \ref{clem-4.1} 
that it   suffices to prove the following inequality
	\begin{equation}\label{ceq-4.7}
 \sup_{r\in(0,1/2)} \left \|\mathcal{T}_r [f] \right \|_{L^1( \mathbb{T})}<\infty 
	\end{equation}
holds for every $f\in X_{\omega}(\mathbb{T})$.

By the assumption that  $(b_1)$ holds   for $p=1$,
we know that for every $f\in X_{\omega}(\mathbb{T})$,
 there exists a positive constant $C_7=C_{7}(f,\omega)$  such that
	\begin{equation}\label{ceq-4.8}
 \mathcal L_{1}[\mathcal{P}[f]](z,w)
	=\int_{0}^{2\pi}\big|\mathcal{P}[f](ze^{i\eta})-\mathcal{P}[f](we^{i\eta})\big |d\eta
\leq
	C_{7}\omega(|z-w|) 
	\end{equation} 
for every $z,w\in  \overline{\mathbb{D}} $.
Letting $z=1-r$ and $w=1$ in \eqref{ceq-4.8} gives that
$$
	\left \|\mathcal{P}_{1-r}[f]-f \right \|_{L^1( \mathbb{T})}
=\int_{0}^{2\pi}\left|\mathcal{P}[f]\big ((1-r)e^{i\eta} \big)-f(e^{i\eta})\right|d\eta
	\leq
	C_{7}\omega(r).
$$
This, together with \eqref{ceq-4.6}  gives that
	\begin{equation*}
 \sup_{r\in(0,1/2)}  	\left \|\mathcal{T}_r [f]   \right\|_{L^1( \mathbb{T})}
\leq C_{7}
	\end{equation*}
whenever  $f\in X_{\omega}(\mathbb{T})$, and, so \eqref{ceq-4.7} holds true.

Now, we deduce from Lemma \ref{clem-4.1}, \eqref{ceq-4.7} and
  \cite[Theorem 2.2]{bre}  that 
$$
 \sup_{r\in(0,1/2)}
		\left  \|\mathcal{T}_r   \right\|_{X_\omega(\mathbb{T})\to L^1(\mathbb{T})}
	<\infty.
$$
By taking 
$$
C_{6}=\sup_{r\in(0,1/2)}
	\|\mathcal{T}_r\|_{X_\omega(\mathbb{T})\to L^1(\mathbb{T})},$$
we see that \eqref{ceq-4.5} holds.
\end{proof}

For every $N\in\mathbb{N}$,
define
\begin{equation}\label{ceq-4.9}
\Phi_N(e^{i \theta})=\text{sgn}\big(\sin (N\theta)\big) 
\end{equation}
on $\mathbb{T}$.
Then we have the following result.

\begin{lem}\label{clem-4.3}
	Let $\omega$ be a majorant.
 Then  for every $N\in\mathbb{N}$,
\begin{equation}\label{ceq-4.10}
	\| \Phi_N\|_{\Lambda_{\omega,1}( \mathbb{T}),s}
	\leq \frac{2\pi^{2}}{\omega(1/N)}.
\end{equation}
 \end{lem}
\begin{proof}

 For every $h\in[-\pi,\pi]$, by calculations, we know that 
 $$|h|\leq\frac{\pi}{2}|e^{ih}-1| .$$
Recall that $\omega(t)$ is increasing on $[0,\infty)$ and $\omega(t)/t$ is non-increasing on $(0,\infty)$.
Then for  every $h\in[-\pi,\pi]$, we obtain 
 $$
\omega(|h|)\leq \omega\left(\frac{\pi}{2}|e^{ih}-1 |\right)
\leq
\frac{\pi}{2}\omega(|e^{ih}-1|).
$$
This  shows that
\begin{equation}\label{ceq-4.11}
 \sup_{0<|h|\leq\pi}
 \frac{\|\Delta_h \Phi_N\|_{L^1( \mathbb{T})}}
 {\omega(|e^{ih}-1|)}
 \leq  \sup_{0<|h|\leq\pi}
 \frac{ \pi\|\Delta_h \Phi_N\|_{L^1( \mathbb{T})}}
 {2\omega(|h|)}.
\end{equation}

In the following, we will divide  into three subcases to estimate the quantity $\|\Delta_h\Phi_N\|_{L^1(\mathbb{T})}$.\smallskip

\noindent {\bf Subcase 2.1.}
Suppose that $h\in[-\pi,- \frac{1}{N}]\cup[ \frac{1}{N},\pi]$.\medskip

In this case, obviously,
\begin{align}\label{ceq-4.12} 
	\|\Delta_h\Phi_N\|_{L^1(\mathbb{T})}
 =\int_{0} ^{2\pi} \left |\Phi_N(e^{i(\theta+h)})-\Phi_N(e^{i\theta} )   \right|d\theta \leq 4\pi.
\end{align}

\noindent {\bf Subcase 2.2.}
Suppose that $h\in(0,\frac{1}{N})$.\medskip

 Observe that
$\Phi_N(e^{i(2\pi+\theta)})=\Phi_N(e^{i \theta})$  for any $\theta \in \mathbb{R}$ 
and the jump points of the mapping $\theta\mapsto \Phi_N(e^{i \theta})$ in $[0,2\pi]$ are 
$\theta_j=\frac{j\pi}{N}$,
where 
$j=0,1,\ldots ,2N$.
Let 
$$
\mathbf{E}_{h}
=\left\{\theta \in \left(-\frac{ \pi}{N},\frac{ 2N-1 }{N}\pi\right]:\Phi_N(e^{i( \theta+h)})\neq\Phi_N(e^{i \theta} ) \right\} 
$$
and
$$
\mathbf{E}^{c}_{h}
=\left\{\theta \in \left(-\frac{ \pi}{N},\frac{ 2N-1 }{N}\pi\right]:\Phi_N(e^{i( \theta+h)})=\Phi_N(e^{i \theta}) \right\} .
$$

 For every $\theta \in \mathbf{E}_{h}$,
we deduce from the facts $h\in(0,\frac{1}{N})$ and $[\theta,\theta+h] \subset (-\frac{\pi}{N}, 2\pi)$
 that there   exists  exactly one $\theta_{j_{0}}$ with $j_{0}\in \{0,1,\ldots,2N-1\}$ 
such that $\theta_{j_{0}} \in[\theta ,\theta +h]$.
This implies that  
$$
\theta\in[\theta_{j_{0}}-h,\theta_{j_{0}}].
$$
Since there are at most $2N $ such intervals  contained in  $\left[-\frac{\pi}{N}, \frac{2N-1}{N}\pi\right]$,
 we have
\begin{align}\label{ceq-4.13} 
\ell(\mathbf{E}_{h})\leq 2Nh,
\end{align}
where $\ell(\mathbf{E}_{h})$ denotes the length of $\mathbf{E}_{h}$.

By \eqref{ceq-4.13}  and the fact that
 $|\Phi_N(e^{i(\theta+h)})-\Phi_N(e^{i\theta} )|\leq 2$ for any $\theta\in \mathbb{R}$,
we find that
\begin{align}\label{ceq-4.14} 
	\|\Delta_h\Phi_N\|_{L^1(\mathbb{T})}
&=\int_{  \mathbf{E}_{h}} |\Phi_N(e^{i(\theta+h)})-\Phi_N(e^{i \theta } )|d\theta+\int_{  \mathbf{E}^{c}_{h}} |\Phi_N(e^{i(\theta+h)})-\Phi_N(e^{i \theta })|d\theta \\\nonumber
&=\int_{  \mathbf{E}_{h}} |\Phi_N(e^{i(\theta+h)})-\Phi_N(e^{i \theta })|d\theta \\\nonumber
&\leq 4Nh.
\end{align}
 
\noindent {\bf Subcase 2.3.}
Suppose that $h\in(-\frac{1}{N},0)$.\medskip

Let 
$$
\mathbf{F}_{h}=\big\{\theta \in [0,2\pi):\Phi_N(e^{i( \theta+h)})\not=\Phi_N(e^{i \theta}) \big\} 
$$
and
$$
\mathbf{F}^{c}_{h}=\big\{\theta \in [0,2\pi):\Phi_N(e^{i( \theta+h)})=\Phi_N(e^{i \theta}) \big\} .
$$

 For every $\theta \in \mathbf{F}_{h}$, 
the facts $h\in(-\frac{1}{N},0)$ and $[\theta+h,\theta] \subset (-\frac{\pi}{N}, 2\pi)$
ensure that there   exists  exactly one $\theta_{k_{0}}$ 
 with $k_{0}\in \{0,1,\ldots,2N-1\}$ 
such that $\theta_{k_{0}} \in[\theta+h ,\theta  ]$, 
 which implies that  
$$
\theta\in[\theta_{k_{0}},\theta_{k_{0}}-h].
$$
Observe that there are at most $2N $ such intervals    contained in  $[0,2\pi]$.
Then $\ell(\mathbf{F}_{h})\leq 2N|h|$.
By  the similar arguments as in \eqref{ceq-4.14} , we get that 
\begin{align}\label{ceq-4.15} 
	\left \|\Delta_h\Phi_N  \right\|_{L^1(\mathbb{T})}
 =\int_{  \mathbf{F}_{h}} |\Phi_N(e^{i(\theta+h)})-\Phi_N(e^{i\theta} )|d\theta 
 \leq 4N|h|.
\end{align}
 
\smallskip
 Now, we are prepared to prove \eqref{ceq-4.10}.
On one hand, we infer from \eqref{ceq-4.14}, \eqref{ceq-4.15} 
and the non-increasing property of the mapping $\omega(t)/t$ that 
\begin{equation*} 
 \sup_{0<|h|<1/N}
 \frac{ \pi\|\Delta_h \Phi_N\|_{L^1(\mathbb{T})}}
 {2\omega(|h|)}
  \leq  \sup_{0<|h|<1/N}
 \frac{ \pi\|\Delta_h \Phi_N\|_{L^1(\mathbb{T})}}
 {2N|h|\omega(1/N)}
 \leq \frac{2 \pi }
 { \omega(1/N)} .
\end{equation*}
On the other hand, by \eqref{ceq-4.12}  and the increasing property of the mapping $\omega$, we have that
\begin{equation*} 
  \sup_{1/N\leq |h|\leq\pi}
 \frac{ \pi\|\Delta_h \Phi_N\|_{L^1(\mathbb{T})}}
 {2\omega(|h|)}
  \leq  
 \frac{ 2\pi^{2} }
 { \omega(1/N)}.
\end{equation*}
These, together with \eqref{ceq-1.4} and \eqref{ceq-4.11}, show    that 
\begin{equation*} 
 \|\Phi_N\|_{\Lambda_{\omega,1}(\mathbb{T}),s}
= \sup_{0<|h|\leq\pi}
 \frac{\|\Delta_h \Phi_N\|_{L^1(\mathbb{T})}}
 {\omega(|e^{ih}-1|)}
  \leq  \sup_{0<|h|\leq\pi}
 \frac{ \pi\|\Delta_h \Phi_N\|_{L^1(\mathbb{T})}}
 {2\omega(|h|)}
 \leq \frac{2 \pi^{2} }
 { \omega(1/N)},
\end{equation*}
which completes the proof.
\end{proof}

By Lemmas \ref{clem-4.2} and \ref{clem-4.3}, we get the following estimate.

\begin{lem}\label{clem-4.4}
Suppose that $\omega$ is a majorant and    $(b_{1})$ holds  for $p=1$.
Then  for every $N\in\mathbb{N}$ and every $r\in(0,1/2)$,
	\begin{equation*}
	 \left  \|\Phi_N-\mathcal{P}_{1-r}[\Phi_N]  \right\|_{L^1( \mathbb{T})}
	\leq
	 C_{6} \omega(r)\frac{2\pi^{2}+ \omega(1)}{\omega(1/N) },
\end{equation*}
where $ C_{6}$ is the constant from Lemma \ref{clem-4.2} and  $\Phi_N $
is the mapping defined in \eqref{ceq-4.9}.

\end{lem}
\begin{proof}
For every  $N\in\mathbb{N}$, $r\in(0,1/2)$ and $\sigma \in(0,1)$,
obviously,
\begin{align}\label{ceq-4.16}
 \left \|\Phi_N-\mathcal{P}_{1-r}[\Phi_N]  \right\|_{L^1( \mathbb{T})} 
\leq&  \left  \| \Phi_N -\mathcal{P}_\sigma[\Phi_N] \right \|_{L^1( \mathbb{T})}
+  \left  \|\mathcal{P}_\sigma[\Phi_N]-\mathcal{P}_{\sigma(1-r)}[\Phi_N] \right\|_{L^1( \mathbb{T})}\\\nonumber
	 &+ \left  \|\mathcal{P}_{\sigma(1-r)}[\Phi_N]-\mathcal{P}_{1-r}[\Phi_N]  \right\|_{L^1( \mathbb{T})}. 
\end{align}
By the semigroup property of the Poisson kernel (cf. \cite[p. 5]{liliu}), 
we know that for every $\sigma ,\tau\in(0,1)$ and $\theta \in \mathbb{R}$,
\begin{align*}
 \mathcal{P}_{\tau}[\mathcal{P}_\sigma[\Phi_N]] (e^{i\theta})
=
\mathcal{P}_{\tau\sigma }[\Phi_N] (e^{i\theta}).
\end{align*} 
Then we infer from Lemma \ref{lem-E} that
\begin{align}\label{ceq-4.17}
 	 \left  \|\mathcal{P}_{\sigma(1-r)}[\Phi_N]-\mathcal{P}_{1-r}[\Phi_N]  \right \|_{L^1( \mathbb{T})}
	&=
	 \left \|\mathcal{P}_{1-r}\big [ \mathcal{P}_\sigma[\Phi_N]-\Phi_N \big] \right   \|_{L^1( \mathbb{T})}\\\nonumber
	&\leq \left \|\mathcal{P}_{\sigma}[\Phi_N]-  \Phi_N  \right \|_{L^1( \mathbb{T})} 
\end{align}
whenever $\sigma \in(0,1)$ and $r\in(0,1/2)$.
Notice that  \cite[p. 114]{ABR92} guarantees that 
\begin{align}\label{ceq-4.18}
\lim_{\sigma\to1^{-} }    \left  \| \mathcal{P}_\sigma[\Phi_N]-\Phi_N  \right\|_{L^1( \mathbb{T})}=0.
\end{align}
Then  for every  $r\in(0,1/2)$, it follows from \eqref{ceq-4.17} and  \eqref{ceq-4.18} that
\begin{align}\label{ceq-4.19}
\lim_{\sigma\rightarrow 1^{-}}   \left \|\mathcal{P}_{\sigma(1-r)}[\Phi_N]-\mathcal{P}_{1-r}[\Phi_N]   \right \|_{L^1( \mathbb{T})}=0.
\end{align} 
Let $\sigma\to1^{-}$ in   \eqref{ceq-4.16}.
We obtain from \eqref{ceq-4.18} and \eqref{ceq-4.19} that for every $r\in(0,1/2)$,
\begin{align}\label{ceq-4.20}
  \left  \|\Phi_N-\mathcal{P}_{1-r}[\Phi_N]   \right\|_{L^1( \mathbb{T})} 
 \leq \liminf_{\sigma   \rightarrow 1^{-}}  \left \|\mathcal{P}_\sigma[\Phi_N]-\mathcal{P}_{\sigma(1-r)}[\Phi_N]  \right \|_{L^1( \mathbb{T})}.
\end{align}

In the following,  
we shall  employ Lemmas \ref{clem-4.2} and \ref{clem-4.3} to estimate the quantity
 $\big\|\mathcal{P}_\sigma[\Phi_N]-\mathcal{P}_{\sigma(1-r)}[\Phi_N] \big\|_{L^1( \mathbb{T})}$
 in terms of $\omega$ as formulated in \eqref{ceq-4.25}.
 For this,   we need to show that $\mathcal{P}_{\sigma}[\Phi_N]\in X_\omega(\mathbb{T})$.
For every $\sigma\in (0,1)$, $N\in \mathbb{N}$ and $\theta\in \mathbb{R}$, since
 $$
 \mathcal{P}_\sigma[\Phi_N](e^{i\theta})=\mathcal{P}[\Phi_N](\sigma e^{i\theta}),
 $$ 
 we see that $\mathcal{P}_\sigma[\Phi_N]\in \mathcal{C}(\mathbb{T})$.
 
 Observe that for every  $\theta,h\in  \mathbb{R}$,  
\begin{align*}
	\qquad\Delta_h(\mathcal{P}_\sigma[\Phi_N])(e^{i\theta})
	&=
	\mathcal{P}_\sigma[\Phi_N](e^{i(\theta+h)})-\mathcal{P}_\sigma[\Phi_N](e^{i\theta})\;\;\;\qquad\qquad\qquad\qquad\text{(by   \eqref{ceq-1.3})}\\
	&=
	\frac{1}{2\pi}
	\int_{0}^{2\pi}
	\mathcal{P}_\sigma(\eta)
	\bigl(\Phi_N(e^{i(  \theta+h-\eta)})-\Phi_N(e^{i( \theta-\eta )})\bigr)\,d\eta \qquad\quad \;\text{(by   \eqref{ceq-3.1})}\\
	&=
	\frac{1}{2\pi}
	\int_{0}^{2\pi}
	\mathcal{P}_\sigma(\eta)\Delta_h \Phi_N(e^{i( \theta-\eta )})\,d\eta\qquad \qquad\qquad\qquad\quad\text{(by   \eqref{ceq-1.3})}\\
	&=
	\mathcal{P}_\sigma[\Delta_h \Phi_N](e^{i\theta}).\qquad\qquad\qquad\qquad\qquad\qquad\;\;\qquad\;\text{(by   \eqref{ceq-3.1})}
\end{align*}
Hence, it follows from  Lemma  \ref{lem-E}  that
\begin{align}\label{ceq-4.21}
	\left\|\Delta_h\big (\mathcal{P}_\sigma[\Phi_N] \big)\right\|_{L^1(\mathbb{T})}
	=
	\left\|\mathcal{P}_\sigma[\Delta_h\Phi_N]\right\|_{L^1(\mathbb{T})}
	\leq
	\left\|\Delta_h\Phi_N  \right\|_{L^1(\mathbb{T})}.
\end{align}
Combining \eqref{ceq-1.4} and  \eqref{ceq-4.21} yields that
\begin{align}\label{ceq-4.22}
 \|\mathcal{P}_\sigma[\Phi_N]\|_{\Lambda_{\omega,1}(\mathbb{T}),s}
 =
  \sup_{0<|h|\leq\pi}
 \frac{\big\|\Delta_h\big (\mathcal{P}_\sigma[\Phi_N] \big)\big\|_{L^1(\mathbb{T})}}
 {\omega(|e^{ih}-1|)} 
 \leq  \sup_{0<|h|\leq\pi}
 \frac{\|\Delta_h\Phi_N\|_{L^1(\mathbb{T})}}
 {\omega(|e^{ih}-1|)} .
\end{align}
Notice that  \eqref{ceq-1.4} and Lemma  \ref{clem-4.3} guarantee   
\begin{align}\label{ceq-4.23}
 \sup_{0<|h|\leq\pi}
 \frac{\|\Delta_h\Phi_N\|_{L^1(\mathbb{T})}}
 {\omega(|e^{ih}-1|)}
=\| \Phi_N\|_{\Lambda_{\omega,1}(\mathbb{T}),s}
	\leq \frac{2\pi^{2}}{\omega(1/N)} .
\end{align}
Substituting  \eqref{ceq-4.23} into \eqref{ceq-4.22} gives that 
\begin{align}\label{ceq-4.24}
 \|\mathcal{P}_\sigma[\Phi_N]\|_{\Lambda_{\omega,1}( \mathbb{T}),s}
 \leq \frac{2\pi^{2}}{\omega(1/N)} .
 \end{align}
Further,  Lemma \ref{lem-E} and the increasing property of $\omega$ give  that  
\begin{align*}
 \|\mathcal{P}_\sigma[\Phi_N]\|_{\max}
 =\|\mathcal{P}_\sigma[\Phi_N]\|_{L^{\infty}(\mathbb{T})}
\leq
 \|\Phi_N\|_{L^{\infty}(\mathbb{T})}
=
1
\leq
\frac{\omega(1)}{\omega(1/N)}.
 \end{align*}
This, together with  \eqref{ceq-4.1} and \eqref{ceq-4.24}, gives that
$\mathcal{P}_\sigma[\Phi_N]\in X_\omega(\mathbb{T})$ and
\begin{equation*}
\|\mathcal{P}_\sigma[\Phi_N]\|_{X_\omega(\mathbb{T})}
	\leq
\frac{2\pi^{2}+ \omega(1)}{\omega(1/N)}.
\end{equation*}
By replacing $f$ with $\mathcal{P}_\sigma[\Phi_N]$ in Lemma \ref{clem-4.2}, we find that
for every $N\in \mathbb{N} $,  $r\in(0,1/2)$ and $\sigma \in (0,1)$,
\begin{align*}
\left \|\mathcal{P}_\sigma[\Phi_N]-\mathcal{P}_{1-r}\big[\mathcal{P}_\sigma[\Phi_N]\big] \right \|_{L^1( \mathbb{T})}
&\leq C_{6} \omega(r)  \left\|\mathcal{P}_\sigma[\Phi_N]  \right\|_{X_\omega(\mathbb{T})}\\\nonumber
&\leq  C_{6} \omega(r)\frac{2\pi^{2}+ \omega(1)}{\omega(1/N) },
	\end{align*}
where  $C_{6} $  is as in  Lemma \ref{clem-4.2}.
 Recall that $\mathcal{P}_{1-r}\big[\mathcal{P}_\sigma[\Phi_N]\big]=\mathcal{P}_{\sigma(1-r)} [\Phi_N] $ on $\mathbb{T}$.
 Then we have that
 \begin{align}\label{ceq-4.25}
\left \|\mathcal{P}_\sigma[\Phi_N]-\mathcal{P}_{\sigma(1-r)} [\Phi_N]\right \|_{L^1( \mathbb{T})}
&\leq  C_{6} \omega(r)\frac{2\pi^{2}+ \omega(1)}{\omega(1/N) }.
	\end{align}

Now, it follows from  \eqref{ceq-4.20}  and \eqref{ceq-4.25} that for every $r\in(0,1/2)$,
\begin{align*} 
\left   \|\Phi_N-\mathcal{P}_{1-r}[\Phi_N]  \right \|_{L^1( \mathbb{T})} 
 \leq \liminf_{\sigma\to 1^{-}}  \left \|\mathcal{P}_\sigma[\Phi_N]-\mathcal{P}_{\sigma(1-r)} [\Phi_N]  \right  \|_{L^1( \mathbb{T})}
 \leq
	 C_{6} \omega(r)\frac{2\pi^{2}+ \omega(1)}{\omega(1/N) },
\end{align*}
as required.
\end{proof}

For every $N\in\mathbb{N}$, $\theta\in \mathbb{R}$ and $f\in L^{1}(\mathbb T)$, let
	$$
	f_{N}(e^{i\theta}) =f(e^{iN\theta}) .
		$$
 Then $f_{N}\in L^{1}(\mathbb T)$.
In the following, we investigate the relationship between the mappings $\mathcal{P}_{r}[f_{N}]$ and $\mathcal{P}_{r^{N}}[f]$.
\begin{lem}\label{clem-4.5}
For every $N\in\mathbb{N}$, $r\in(0,1)$, $ \theta\in\mathbb R$ and  $f\in L^{1}(\mathbb{T})$, 
	\begin{equation}\label{ceq-4.26}
		\mathcal{P}_{r}[f_{N}](e^{i\theta})
		=
		\mathcal{P}_{r^{N}}[f](e^{iN\theta}).
	\end{equation}

\end{lem}
\begin{proof}
In order to prove the lemma, first, we need to express the Fourier coefficients of $ f_{N} $ and  $ \mathcal{ P}_r[f_{N}] $ 
in terms of those of $ f $ as formulated in \eqref{ceq-4.27} and \eqref{ceq-4.29} below.

On the one hand, for  $k\in\mathbb{Z}$,  we infer from \eqref{ceq-2.1} that
 \begin{align*}
 \widehat{f_{N}}(k)
 &=\frac{1}{2\pi}\int_{0}^{2\pi} f(e^{iN \vartheta })e^{-ik \vartheta }\,d \vartheta \\
& = \frac{1}{2\pi N} \int_{0}^{2\pi N} f(e^{i\theta})e^{-ik\theta/N}\,d\theta\\
  &= \frac{1}{2\pi N}
 \sum_{j=0}^{N-1}
 \int_{2\pi j}^{2\pi(j+1)}
 f(e^{i\theta})e^{-ik\theta/N}\,d\theta,
 \end{align*}
where
\begin{align*} 
\int_{2\pi j}^{2\pi(j+1)}
 f(e^{i\theta})e^{-ik\theta/N}\,d\theta
 &=\int_{0}^{2\pi}
 f(e^{i(\theta+2\pi j)})e^{-ik(\theta+2\pi j)/N}\,d\theta\\
& =\int_{0}^{2\pi}
			f(e^{i\theta})e^{-ik(\theta+2\pi j)/N}\,d\theta  
 \end{align*}
 for every $j\in\{0,1,\ldots,N-1\}$.
Then it follows that
		\begin{align*}
			\widehat{f_{N}}(k)
			&=
			\frac{1}{2\pi N}
			\sum_{j=0}^{N-1}
			\int_{0}^{2\pi}
			f(e^{i\theta})e^{-ik(\theta+2\pi j)/N}\,d\theta\\
			&=
			\frac{1}{N}
			\left(
			\sum_{j=0}^{N-1}e^{-2\pi i kj/N}
			\right)
			\frac{1}{2\pi}
			\int_{0}^{2\pi}
			f(e^{i\theta})e^{-ik \theta/N}\,d\theta.
		\end{align*}
	Notice that
		$$
		\sum_{j=0}^{N-1}e^{-2\pi i kj/N}
		=
		\begin{cases}
			N, & \text{if }N\mid k,\\
			0, & \text{if }N\not\mid k,
		\end{cases}
		$$
where $N\mid k$ means that  $ \frac{k}{N}$ is an integer.
Then
		\begin{equation}\label{ceq-4.27}
			\widehat{f_{N}}(k)
			=
			\begin{cases}
				\widehat f(m), & \text{if $k=mN$ for some $m\in \mathbb{Z}$},\\
				0, & \text{if }N\not\mid k 
			\end{cases}
		\end{equation}
(cf. \cite[p. 7]{KY2004}).
		
On the other hand, for every $N\in\mathbb{ N}$, $r\in(0,1)$ and $\theta \in \mathbb{R}$, 
by  \eqref{P_r[f]-Fourier} and the fact  that
 $f\in L^{1}(\mathbb{T})$,
  we know that
 \begin{equation}\label{ceq-4.28}
\mathcal{ P}_r[f_{N}](e^{i \theta})=\sum_{k=-\infty}^{\infty} r^{|k|}\widehat{f_{N}}(k) e^{ik\theta}.   
	\end{equation}
This, together with \eqref{ceq-4.27}, implies that  for every $k\in \mathbb{Z}$,
		 		\begin{equation}\label{ceq-4.29}
			\widehat{\mathcal{ P}_r[f_{N}]}(k)
			=
			\begin{cases}
				(r^{N})^{|m|}\widehat f(m), & \text{if $k=mN$ and $m\in \mathbb{Z}$},\\
				0, & \text{if }N\not\mid k.
			\end{cases}
		\end{equation}
		
In the following, we will   exploit \eqref{ceq-4.27} and \eqref{ceq-4.29} to prove \eqref{ceq-4.26}.
For this, we let
$$
		h_{N,r,1}(e^{i \theta}) = \mathcal{P}_{r^{N}}[f](e^{i \theta}) \;\text{and}\;
h_{N,r,2}(e^{i \theta})= \mathcal{P}_{r^{N}}[f](e^{iN \theta})  ,
		$$
where $N\in \mathbb{N}$, $r\in(0,1)$ and $\theta\in \mathbb{R}$. 
Obviously,  $h_{N,r,1},h_{N,r,2} \in \mathcal{C}(\mathbb{T})$.
Replace $f $ and $f_{N}$ with $h_{N,r,1}$
 and   $h_{N,r,2}$ in  \eqref{ceq-4.27}, respectively.
Then we obtain that for every $k\in \mathbb{Z}$,
		\begin{equation*}
			\widehat{h_{N,r,2}}(k)
			=
			\begin{cases}
				\widehat{h_{N,r,1}}(m), &\text{if $k=mN$ for some  $m\in \mathbb{Z}$},\\
				0, & \text{if }N\not\mid k.
			\end{cases}
		\end{equation*}
Since \eqref{ceq-4.28} for $N=1$ implies that 
		$$
		\widehat{h_{N,r,1}}(m)=\widehat{\mathcal{P}_{r^{N}}[f]}(m)
		=
		(r^{N})^{|m|}\widehat{f}(m) ,
		$$ 
we see that
		\begin{equation}\label{ceq-4.30}
		\widehat{h_{N,r,2}}(k)
			=
			\begin{cases}
				(r^{N})^{|m|}\widehat f(m), &\text{if $k=mN$ for some  $m\in \mathbb{Z}$},\\
				0, & \text{if }N\not\mid k.
			\end{cases}
		\end{equation}
	Comparing \eqref{ceq-4.29} and
		\eqref{ceq-4.30}, we see that  for every $k\in \mathbb{Z}$,
		$$
	\widehat{\mathcal{ P}_r[f_{N}]}(k)=	\widehat{h_{N,r,2}}(k).
		$$
Then the uniqueness theorem of Fourier series (cf. \cite[p. 12]{KY2004})	and the continuity of the mappings $\mathcal{ P}_r[f_{N}]$ and $h_{N,r,2}$ 
imply that for every $\theta\in \mathbb{R}$,
$$ 
\mathcal{P}_{r}[f_{N}](e^{i\theta}) =h_{N,r,2}(e^{i\theta})= \mathcal{P}_{r^{N}}[f](e^{i  N\theta })  .
$$
 This proves \eqref{ceq-4.26}.
\end{proof}

By Lemma \ref{clem-4.5}, we obtain the following result.

\begin{lem}\label{clem-4.6}
For     $r\in(0,1/2)$ and $N_{r}\in \big(0,\log_{1-r}\frac{1}{2} \big) \cap \mathbb{N}$,
we have
	\begin{equation}\label{ceq-4.31}
		\left\|\Phi_{N_{r}}-\mathcal{P}_{1-r}[\Phi_{N_{r}}] \right\|_{L^1( \mathbb{T})}
		\geq
		\frac{ r N_{r}}{5\pi}\log\frac{1}{2r N_{r}},
	\end{equation}
where $\Phi_{N_{r}}$ is the mapping defined in \eqref{ceq-4.9}.
\end{lem}

\begin{proof}
By \eqref{ceq-4.9}, we know that for every $N\in  \mathbb{N}$ and $\vartheta\in \mathbb{R}$,
\begin{equation*}
	\Phi_N(e^{i\vartheta})=\text{sgn}(\sin (N\vartheta))=\Phi_1(e^{iN\vartheta}).
\end{equation*}
Since  $\Phi_1\in L^{1}(\mathbb T)$, we deduce from Lemma \ref{clem-4.5}  that
 for  every $\rho\in(0,1)$ and $ \vartheta\in\mathbb R$,
$$
\mathcal{P}_{\rho}[\Phi_N](e^{i\vartheta})
=
\mathcal{P}_{\rho^N}[\Phi_1](e^{iN\vartheta}).
$$
From  the above equalities, we get  
\begin{align}\label{ceq-4.32}
\left\|\Phi_N- \mathcal{P}_{\rho}[\Phi_N] \right\|_{L^1( \mathbb{T})}
&=	\int_0^{2\pi}	\left|\Phi_1(e^{iN\vartheta})-	\mathcal{P}_{\rho^N}[\Phi_1](e^{iN\vartheta})	\right|\,d\vartheta\\\nonumber
&=\frac{1}{N}	\int_0^{2N\pi}	\left|\Phi_1(e^{i\theta})-\mathcal{P}_{\rho^N}[\Phi_1](e^{i\theta})\right|\,d\theta.
	\end{align}
Observe that for  every $\theta\in \mathbb{R}$,
\begin{align}\label{ceq-4.33}
\Phi_1(e^{i(\theta+2\pi)})-\mathcal{P}_{\rho^N}[\Phi_1](e^{i(\theta+2\pi)})
=\Phi_1(e^{i\theta})-\mathcal{P}_{\rho^N}[\Phi_1](e^{i\theta}).
\end{align}
Then  \eqref{ceq-4.32} and \eqref{ceq-4.33} imply that
\begin{equation}\label{ceq-4.34}
\left \|\Phi_N-\mathcal{P}_{\rho}[\Phi_N] \right \|_{L^1( \mathbb{T})}
=
\left \|\Phi_1-\mathcal{P}_{\rho^N}[\Phi_1]   \right\|_{L^1( \mathbb{T})},
	\end{equation}
where $\rho\in (0,1)$ and $N\in \mathbb{N}$.

To continue the proof, we still need the following claim.
\smallskip

\noindent {\bf Claim 2.1.}
 For any $\rho\in(1 /2,1)$,
	\begin{align*} 
		\left\|\Phi_1-\mathcal{P}_\rho[\Phi_1] \right\|_{L^1(\mathbb{T})}
		\geq
		 \frac{2(1-\rho)}{ 5\pi}\log\frac{1}{2(1-\rho)}.
\end{align*}
 
Indeed, it follows from \eqref{ceq-1.5}, \eqref{ceq-3.1} and \eqref{ceq-4.9} that  for  every $\rho\in(1/2,1)$, 
 \begin{align}\label{ceq-4.35}
\left \|  \Phi_1-\mathcal{P}_\rho[\Phi_1]  \right\|_{L^1(\mathbb{T})} 
   = &\int_{-\pi}^{\pi}\left| \frac{1}{2\pi} \int_{-\pi}^{\pi}\frac{1-\rho^{2}}{|1-\rho e^{i(\theta-\eta)}|^{2}}
 \big( \Phi_{1}(e^{i \theta }) - \Phi_{1}(e^{i \eta})\big) d\eta\right|d \theta\\\nonumber
=&\int_{0}^{\pi}\left( \frac{1}{2\pi} \int_{-\pi}^{\pi}\frac{1-\rho^{2}}{|1-\rho e^{i(\theta-\eta)}|^{2}}
 \big( 1 - \Phi_{1}(e^{i\eta })\big) d\eta \right)d \theta\\\nonumber
 &+  \int_{-\pi}^{0}\left( \frac{1}{2\pi} \int_{-\pi}^{\pi}\frac{1-\rho^{2}}{|1-\rho e^{i(\theta-\eta)}|^{2}}
 \big(1+ \Phi_{1}(e^{i \eta })\big) d\eta\right)d \theta .
 \end{align}
  Since  \eqref{ceq-4.9} guarantees that for  every $\theta\in \mathbb{R}$,
	\begin{align*}
		 \frac1{2\pi}
		\int_{-\pi}^{ \pi}
		\frac{1-\rho^{2}}{|1-\rho e^{i(\theta-\eta)}|^{2}}
		\big(1- \Phi_{1}(e^{i\eta})\big)\,d\eta
		 =
		\frac{1}{\pi}
		\int_{-\pi}^{0}
		\frac{1-\rho^{2}}{|1-\rho e^{i(\theta-\eta)}|^{2}}\,d\eta
	\end{align*}
and
	\begin{align*}
		 \frac1{2\pi}
		\int_{-\pi}^{ \pi}
		\frac{1-\rho^{2}}{|1-\rho e^{i(\theta-\eta)}|^{2}}
	\big(1+ \Phi_{1}(e^{i \eta })\big)\,d\eta
		 =
		\frac{1}{\pi}
		\int_{0}^{\pi}
		\frac{1-\rho^{2}}{|1-\rho e^{i(\theta-\eta)}|^{2}}\,d\eta,
	\end{align*}
we deduce from \eqref{ceq-4.35} that
\begin{align}\label{ceq-4.36}
  \big\|   \Phi_1-\mathcal{P}_\rho[\Phi_1]  \big\|_{L^1(\mathbb{T})}   
  =& \frac{1}{\pi} \int_{0 }^{\pi}\int_{-\pi}^{0}
 \frac{1-\rho^{2}}{|1-\rho e^{i(\theta-\eta)}|^{2}}\,d\eta d\theta  +\frac{1}{\pi}
\int_{-\pi}^{0} \int_{0}^{\pi}
 \frac{1-\rho^{2}}{|1-\rho e^{i(\theta-\eta)}|^{2}}\,d\eta d\theta  \\\nonumber
= &
 \frac{2}{ \pi}
\int_{0}^{\pi}\int_{-\pi}^{0}
\frac{1-\rho^{2}}{|1-\rho e^{i(\theta-\eta)}|^{2}}\,d\eta\,d\theta .
\end{align}
	
Observe that for  every $\rho\in(1/2,1)$, $\theta \in (1-\rho,1/2)\subset[0,\pi] $  and 
$\eta\in(-\theta,0)\subset[-\pi,0]$,  
 $$
 \theta-\eta\in(\theta,2\theta) 
\;\;\text{and} \;\; 1-\rho\in  (0,\theta)\subset  (0,1/2).
 $$
Then 
$$
4\rho\sin^2 \frac{ \theta-\eta}{2} \leq ( \theta-\eta)^2\leq 4\theta^2,
$$ 
and so,
$$
\frac{1-\rho^{2}}{|1-\rho e^{i(\theta-\eta)}|^{2}}
	=
	\frac{1-\rho^2}
	{(1-\rho)^2+4\rho\sin^2  \frac{ \theta-\eta}{2}  }\geq \frac{1-\rho }{(1-\rho)^{2}+4\theta^2}\geq \frac{1-\rho}{5\theta^{2}}.
$$
 Substituting this estimate into \eqref{ceq-4.36} gives that 
	\begin{align*}
		\left\|\Phi_1-\mathcal{P}_\rho[\Phi_1]\right\|_{L^1(\mathbb{T})}
		\geq
		 \frac{2(1-\rho)}{ 5\pi}
		\int_{1-\rho}^{1/2}
		\int_{-\theta}^{0}
		\frac1{\theta^2}\,d\eta\,d\theta
	 =
		 \frac{2(1-\rho)}{ 5\pi}\log\frac{1}{2(1-\rho)}.
\end{align*}
Hence, Claim 2.1 holds true.

\medskip

By \eqref{ceq-4.34} and Claim 2.1, we find that
 for  every $\rho\in(1/2,1)$ and $N_{\rho}\in \big (0,\log_{\rho}\frac{1}{2} \big) \cap \mathbb{N}$,
 \begin{align}\label{ceq-4.37}
\left\|\Phi_{N_{\rho}}-\mathcal{P}_{\rho}[\Phi_{N_{\rho}}] \right \|_{L^1( \mathbb{T})}
 = \left \|\Phi_1-\mathcal{P}_{\rho^{N_{\rho}}}[\Phi_1]  \right\|_{L^1( \mathbb{T})}
\geq \frac{2(1-\rho^{N_{\rho}})}{ 5\pi}\log\frac{1 }{2(1-\rho^{N_{\rho}})}.
 \end{align}
Set $\rho=1-r$. Then $r\in (0,1/2)$.
Obviously, in order to prove \eqref{ceq-4.31}, we still need to obtain the
upper and lower bounds of the quantity of $1-(1-r)^{N_{r}}$ in terms of $rN_{r}$, where 
 $N_{r}\in \big(0,\log_{1-r}\frac{1}{2} \big)\cap \mathbb{N} $.

 By Bernoulli's inequality  (cf. \cite[p. 9]{caro})  and  calculations, we know that for  every $r\in (0,1/2)$ and $N\in \mathbb{N}$,
 \begin{equation}\label{ceq-4.38}
 (1-r)^{N}\geq 1-r N
 \;\;\text{and}\;\;
 (1-r)^{N }\leq e^{-rN} .
 \end{equation}
Further, since $1-e^{-x}-\frac{x}{2} >0$ when $ x\in (0,\log 2)$ and
$$
0<rN_{r}<r\log_{1-r}\frac{1}{2}=\frac{-r\log 2}{\log(1-r)}< \log2,
$$
we obtain that
	\begin{equation}\label{ceq-4.39}
 e^{-rN_{r}}\leq 1-\frac{rN_{r}}{2}.
	\end{equation}
Then we deduce from \eqref{ceq-4.38} and \eqref{ceq-4.39} that
for every  $r\in(0,1/2)$ and $N_{r}\in\big (0,\log_{1-r}\frac{1}{2} \big) \cap \mathbb{N}$,
	\begin{equation}\label{ceq-4.40}
 \frac{rN_{r} }{2}
\leq 1-e^{-rN_{r}}
 \leq1-(1-r)^{N_{r}}
 \leq rN_{r}.
	\end{equation}
	
Now, the combination of \eqref{ceq-4.37} and \eqref{ceq-4.40} implies that
for every    $r\in(0,1/2)$ and $N_{r}\in \mathbb{N}  \cap \big (0,\log_{1-r}\frac{1}{2}  \big)$,
 \begin{align*} 
\left \|\Phi_{N_{r}}-\mathcal{P}_{1-r}[\Phi_{N_{r}}] \right \|_{L^1 (\mathbb{T})}
 \geq \frac{2(1-(1-r)^{N_{r}})}{ 5\pi}\log\frac{1 }{2(1-(1-r)^{N_{r}})}
 \geq  \frac{r{N_{r}} }{ 5\pi}\log\frac{1 }{2r{N_{r}}  }.
 \end{align*}
 This proves \eqref{ceq-4.31}, and, so the proof is
complete.
\end{proof}

We now proceed with the proof of $(c_{2})$ for the case $p=1$.
For convenience, we let 
\begin{equation}\label{ceq-4.401}
Q(t)=\frac{\omega(t)}{t}
\end{equation}
 on $(0,\infty)$.
In the following,
we shall divide the proof into two steps. 
In the first step, we show that there exists  a constant  $A\in(1,\infty)$  
such that for every $j\in \mathbb{N}$ and $r\in(0,A^{-j})$,
\begin{equation}\label{ceq-4.41}
	Q(A^{j}r) \leq
 2^{ -j}Q(r).
\end{equation}
In the second step, 
by applying the   existence of $A$ to estimate the upper bound of 
$ \int_{\delta}^{\pi}	\frac{Q(t)}{t } dt$, 
we finish  the proof of the theorem.

\medskip

\noindent {\bf Step 2.1.}
The   existence of $A$.
 
\smallskip
 
By Lemmas \ref{clem-4.4} and  \ref{clem-4.6}, we find that 
for every  $r\in(0,1/2)$ and $N_{r}\in \big(0,\log_{1-r}\frac{1}{2}\big) \cap \mathbb{N}$,
 \begin{equation*}  
	\frac{rN_{r}}{5\pi}\log\frac{1}{2r N_{r} }
\leq \left\|\Phi_{N_{r}}-\mathcal{P}_{1-r}[\Phi_{N_{r}}] \right\|_{L^1(\mathbb{ T})}
	\leq
	 C_{6} \omega(r)\frac{2\pi^2+ \omega(1)}{\omega(1/N_{r} )},
\end{equation*}
 where $ C_{6}=C_{6}( \omega)$ is the constant from Lemma \ref{clem-4.2}.
This  implies that
for every   $r\in(0,1/2)$ and $N_{r}\in \big(0,\log_{1-r}\frac{1}{2}\big) \cap \mathbb{N}$,
\begin{equation}\label{ceq-4.42}
	Q(1/N_{r})
	\log\frac{1}{2rN_{r}}
	\leq
	C_{8}Q(r) ,
\end{equation}
where $C_{8}=5\pi C_{6}\big(2\pi^2+ \omega(1)\big)>0$.

Let 
\begin{equation}\label{ceq-4.43}
 A=4e^{2C_{8}}+1.
\end{equation} 
For every  $\rho\in (0, \frac{1}{A})$,
from the inequality 
$$
\frac{2}{A\rho}-\frac{1}{A\rho} >1,
$$
we find  that there must exist a positive integer $n_{\rho} \in\left[\frac1{A\rho},\,\frac2{A\rho}\right]$,
which implies that
\begin{equation}\label{ceq-4.44}
\log\frac{1}{2\rho n_{\rho} }\geq\log\frac{  A}{4}> 2C_{8}>0
\;\;\text{and}\;\;
Q(A\rho)\leq Q(1/n_{\rho}).
\end{equation}

Notice that $\frac{1}{\rho}\log(1-\rho)$ is strictly decreasing in $(0,\frac{1}{A}]$ and 
$\log \big(1- \frac{1}{A} \big)>\frac{1}{2} \log \frac{1}{2}$.
Then for   $\rho\in (0, \frac{1}{A})$, we have
$$
\frac{\log(1-\rho)}{\rho}>A\log\left(1-\frac{1}{A}\right)>\frac{A}{2}\log \frac{1}{2}.
$$
This yields that
$$
\frac{2\log(1-\rho)}{A\rho}>\log \frac{1}{2},
$$
and so,
\begin{equation}\label{ceq-4.45}
 n_{\rho} \in\left[\frac1{A\rho},\,\frac2{A\rho}\right]  \subset  \left(0,\log_{1-\rho}\frac{1}{2} \right) .
\end{equation} 

By \eqref{ceq-4.42}, \eqref{ceq-4.44} and \eqref{ceq-4.45},
we see that for every  $\rho\in (0, \frac{1}{A})\subset (0,\frac{1}{2})$ 
and  $n_{\rho} \in\left[\frac1{A\rho}, \frac{2}{A\rho}\right] \subset \left(0,\log_{1-\rho}\frac{1}{2} \right)$,
\begin{align*}
	Q(A\rho)
	\leq
 Q(1/n_{\rho})\leq
	\frac{C_{8}}{\log\frac{1}{2\rho n_{\rho}}}Q(\rho)
	<
	\frac{  Q(\rho)}{2}.
\end{align*}
This implies that  for every  $j\in \mathbb{N}$ and $r\in(0,A^{-j})$,
\begin{equation*}
	Q(A^j r)\leq Q(A^{j-1}r)/2\leq \cdots\leq  2^{1-j}Q(A r) 
	\leq
	2^{-j}Q(r).
\end{equation*}
Hence  the constant  $A$ defined in \eqref{ceq-4.43} satisfies \eqref{ceq-4.41}, and
so this $A$ is what we need.

\medskip
 \noindent {\bf Step 2.2.}
The   estimate of $ \int_{\delta}^{\pi}	\frac{Q(t)}{t }\,dt$ with $\delta\in(0,\pi]$.
\smallskip

In the following,   we divide the proof into two subcases to estimate the quantity
 $ \int_{\delta}^{\pi}	\frac{Q(t)}{t } dt$.\smallskip

\noindent {\bf Subcase 2.4.}
Suppose that $\delta\in[ \frac{1}{2A},\pi]$.\medskip

In this case,   we infer from the non-increasing property of $\omega(t)/t$ and \eqref{ceq-4.401} that
\begin{equation}\label{ceq-4.46}
 \int_{\delta}^{\pi}
\frac{Q(t)}{t }\,dt
\leq Q(\delta)\int_{\delta}^{\pi}
\frac{1}{t }\,dt
\leq
Q(\delta)\int_{  \frac{1}{2A} }^\pi
\frac{1}{t}\,dt
=Q(\delta) \log(2A\pi).
\end{equation}

\smallskip

\noindent {\bf Subcase 2.5.}
Suppose that $\delta\in(0,\frac{1}{2A})$.\medskip

In this case, 
we will estimate  the quantities $\int_{\delta}^{\frac{1}{2A}}\frac{Q(t)}{t}dt$ 
and $\int_{\frac{1}{2A}}^{\pi}\frac{Q(t)}{t}dt$, respectively.
 
\noindent
$(\mathfrak{A})$ {\bf Estimate on   $\int_{\delta}^{ \frac{1}{2A}}\frac{Q(t)}{t}dt$.}
Let $\delta\in(0, \frac{1}{2A})$ and choose an integer $m_{\delta}\in \mathbb{N} $ such that
\begin{equation}\label{ceq-4.47}
A^{m_{\delta}-1}\delta\leq   \frac{1}{2A}<A^{m_{\delta}}\delta.
\end{equation}
  Then
$$
\int_{\delta}^{\frac{1}{2A}}\frac{Q(t)}{t}dt
\leq
\sum_{k=0}^{m_{\delta}-1}
\int_{A^{k}\delta}^{A^{k+1}\delta}
\frac{Q(t)}{t}dt
\leq
\sum_{k=0}^{m_{\delta}-1}
Q(A^{k}\delta)
	\int_{A^{k}\delta}^{A^{k+1}\delta}
	\frac{dt}{t}
=\log A\sum_{k=0}^{m_{\delta}-1}Q(A^{k}\delta).
$$
Since \eqref{ceq-4.47} implies that
$\delta\in\big (0,\frac{1}{2A^{m_{\delta}}} \big] \subset  \big (0,\frac{1}{A^{k}} \big )$ for every $k\in\{ 0,1,\ldots,m_{\delta}-1\}$, 
 we infer from \eqref{ceq-4.41} and the above inequalities that 
  \begin{align}\label{ceq-4.48}
\int_{\delta}^{\frac{1}{2A}}\frac{Q(t)}{t}dt
\leq\log A \sum_{k=0}^{m_{\delta}-1}Q(A^{k}\delta)
\leq Q(\delta)\log A\sum_{k=0}^{m_{\delta}-1}2^{-k}
\leq2 Q(\delta) \log A.
\end{align}

\noindent
$(\mathfrak{B})$ {\bf Estimate on   $\int_{ \frac{1}{2A}}^{\pi}\frac{Q(t)}{t}dt$.}
For   $\delta\in \big (0,\frac{1}{2A} \big)$ and $t\in\big [\frac{1}{2A},\pi\big]$,
obviously,
$$
Q(t)\leq Q\left(\frac{1}{2A}\right)\leq Q(\delta).
$$
Then 
 \begin{equation}\label{ceq-4.49}
 \int_{\frac{1}{2A}}^{\pi} \frac{Q(t)}{t}\,dt
	\leq Q(\delta)\int_{\frac{1}{2A}}^{\pi} \frac{1}{t}\,dt=Q(\delta)\log(2A\pi).
\end{equation}

It follows from \eqref{ceq-4.48} and \eqref{ceq-4.49}  that for   $\delta\in(0, \frac{1}{2A})$,
 \begin{equation}\label{ceq-4.50}
	 \int_{\delta}^{\pi} \frac{Q(t)}{t}\,dt
= \int_{\delta}^{\frac{1}{2A}}\frac{Q(t)}{t}dt+\int_{\frac{1}{2A}}^{\pi} \frac{Q(t)}{t}\,dt
	\leq \big( 2 \log A +\log(2A\pi) \big)Q(\delta).
\end{equation}
Combining \eqref{ceq-4.50} with  \eqref{ceq-4.46} shows that  
\begin{equation*} 
 \int_{\delta}^\pi \frac{Q(t)}{t }\,dt 
\leq  \big(2 \log A+ \log (2A\pi)\big)Q(\delta),  \quad\delta\in(0,\pi].
\end{equation*}
This 
implies that  $(b_2)$ also holds for $p=1$. 
The proof of this theorem is complete.
 \qed

\bigskip
{\bf Data Availability} Our manuscript has no associated data.

{\bf Conflict of interest} The authors declare that they have no conflict of interest.

\bigskip

{\rm $\mathbf{Acknowledgments:}$}	
The second  author was partly supported by
the Scientific Research Fund of Hunan Provincial Education Department (grant no. 25A0086) 
 and the construct program of the key discipline in Hunan Province.
 The first author was partly supported by the
National Science Foundation of China (grant no. 12571080).
The third author was partly supported by JSPS KAKENHI (grant no. JP22K03363).
The  fourth author was partly supported by the National Science Foundation of China
	(grant nos. 12371071 and  12571081).

\end{document}